\documentclass[12pt]{article}

\usepackage{style}
\usepackage{macros}
\usepackage{amsmath}
\usepackage{subfiles}

\usepackage[a4paper,margin=1.2in]{geometry}
\usepackage{fancyhdr}
\title{Higher-dimensional Virasoro algebras}
\author{Zhengping Gui and Brian R. Williams}
\date{\today}

\newcommand{\SP}{\mathbb{P}^{S^1}}
\newcommand{\1}{{\textbf{\texttt{\tiny{1}}}}}
\newcommand{\2}{\textbf{\texttt{\tiny{2}}}}
\newcommand{\jj}{\textbf{\texttt{\tiny{j}}}}

\newcommand{\ii}{\textbf{\texttt{\tiny{i}}}}
\newcommand{\kkk}{\textbf{\texttt{\tiny{k}}}}

\newcommand{\s}{\textbf{\texttt{\tiny{s}}}}
\newcommand{\ddd}{\textbf{\texttt{\tiny{d}}}}
\newcommand{\smallhollowcolon}{%
  \vcenter{\offinterlineskip
    \ialign{##\cr
      \tiny$\circ$\cr
      \noalign{\kern0.4ex}
      \tiny$\circ$\cr
    }%
  }%
}

\newcommand{\dd}{\textbf{\texttt{\tiny{d}}}}

\newcommand{\permute}{\tau}

\usepackage{simpler-wick}
\usepackage{simpler-wick}
\usetikzlibrary{decorations.markings}
\tikzset{W->-/.style={decoration={
  markings,
  mark=at position 0.5*\pgfdecoratedpathlength+2pt with
  {\draw[-latex] (-2pt,0pt) -- (1pt,0pt);}},postaction={decorate}},
  W-<-/.style={decoration={
  markings,
  mark=at position 0.5*\pgfdecoratedpathlength with
  {\draw[latex-] (-2pt,0pt) -- (1pt,0pt);}},postaction={decorate}}
  }
\pgfkeys{
  /simplerwick/.cd,
  arrows/.store in=\LstWickArrows,
  arrows/.initial={-,-,-,-,-,-,-,-,-}, % the # of contractions is bounded by 9
  arrows={-,-,-,-,-,-,-,-,-},
  positions/.store in=\LstWickPositions,
  positions={+1,+1,+1,+1,+1,+1,+1,+1,+1},
  positions/.initial={+1,+1,+1,+1,+1,+1,+1,+1,+1},
}

\makeatletter
\newcounter{Wick@up}
\newcounter{Wick@down}
\def\swick@end#1#2{
  \swick@setfalse@#1
  \tikzexternaldisable
  \begin{tikzpicture}[remember picture, baseline=(swick-close#1.base)]
    \node[use as bounding box, inner sep=0pt, outer sep=0pt] (swick-close#1) {$\displaystyle #2$};
  \end{tikzpicture}
  \tikz[remember picture, overlay]
{
\xdef\myW@style{\empty}
\foreach \W@X[count=\W@C] in \LstWickArrows
{\ifnum\W@C=#1
\xdef\myW@style{\W@X}
\fi}
\ifx\myW@style\empty
\PackageWarning{simpler-wick}{%
The list arrows has not enough entries!%
}{}
\xdef\myW@style{-}
\fi
\xdef\myW@pos{-77}
\foreach \W@X[count=\W@C] in \LstWickPositions
{\ifnum\W@C=#1
\xdef\myW@pos{\W@X}
\fi}
\ifnum\myW@pos=-77
\PackageWarning{simpler-wick}{%
The list positions has not enough entries!%
}{}
\xdef\myW@pos{+1}
\fi
\ifnum\myW@pos=-1
    \draw[\myW@style] ($(swick-open#1.south) + (0, -3pt)$) 
          -- ($(swick-open#1.base) + (0, -\swick@offset) + \theWick@down*(0, -\swick@sep)$) 
          -- ($(swick-close#1.base) + (0, -\swick@offset) + \theWick@down*(0, -\swick@sep)$) 
          -- ($(swick-close#1.south) + (0, -3pt)$);
\stepcounter{Wick@down}
\else
\stepcounter{Wick@up}
    \draw[\myW@style] ($(swick-open#1.north) + (0, 3pt)$) 
          -- ($(swick-open#1.base) + (0, \swick@offset) + \theWick@up*(0, \swick@sep)$) 
          -- ($(swick-close#1.base) + (0, \swick@offset) + \theWick@up*(0, \swick@sep)$) 
          -- ($(swick-close#1.north) + (0, 3pt)$);
\fi}
  \tikzexternalenable}
\def\wick@[#1]#2{\setcounter{Wick@up}{0}
\setcounter{Wick@down}{-1}
  \ifmmode
    \begingroup
    \pgfkeys{
        simplerwick,
        #1}
    \swick@cond@reset
    \swick@count=0
    \def\swick@max{0}
    \def\c{\swick@smart}
    #2
    \dimen0=\swick@sep
    \multiply\dimen0 by \swick@max
    \advance\dimen0 by \swick@offset
    \vbox to \dimen0{}
    \swick@cond@any{
      \PackageWarning{simpler-wick}{%
        I have reached the end of \protect\wick\space with some unclosed
        contractions%
      }{}
    }{}
    \endgroup
  \else
    \PackageWarning{simpler-wick}{%
      \protect\wich\space has been called outside a math environment, this will
      be ignore%
    }
  \fi
}

\makeatother

\begin{document}
\maketitle

\abstract{
  We classify central extensions of the dg Lie algebra of derived global sections of the tangent sheaf on the
punctured, formal $d$-disk for $d > 2$. 
The case of $d=2$ has been handled in \cite{GWvir2d}.
We also prove a local, universal form of the Grothendieck--Riemann--Roch theorem for families of $d$-dimensional complex
varieties.}
% ----------------------------------------------------------------

%\tableofcontents

\section{Introduction}

Let $\mathring{D}^d$ denote the punctured $d$-dimensional formal disk.
For $d > 1$ this is not a formal affine scheme.
The main object of study in this paper is the derived global sections of the tangent sheaf 
\begin{equation}\label{}
  \lie{witt}_d \simeq \R \Gamma(\mathring{D}^d , \sT) 
\end{equation}
on the punctured $d$-disk.
We refer to this dg Lie algebra as the $d$-dimensional Witt algebra.
We write equivalence, rather than definition, since we use $\lie{witt}_d$ to refer to a concrete dg model for the derived global
sections which was introduced in \cite{FHK,GWWchiral} motivated by the theory of Jouanolou torsors.
See \cite{GWvir2d} for more details on the definition of $\lie{witt}_d$; we review its features in section
\ref{s:background}.

The first main result of this paper is the classification of central extensions of this dg Lie algebra.
We will show, see theorem \ref{thm:local}, that the second Lie algebra (hyper) cohomology of $\lie{witt}_d$ is $p(d+1)-1$
dimensional where $p(k)$ is the number of partitions of the integer $k$.
In fact, we use a universal Chern--Weil homomorphism \cite{GWvir2d}:
\begin{equation}\label{}
  \mathsf{cw}_d \colon \C[\op{ch}_1,\op{ch}_2, \ldots, \op{ch}_{d}] \to \HH^2_{Lie} (\lie{witt}_d) 
\end{equation}
which we will show is an isomorphism when restricted to polynomials in the $\op{ch}_i$ which are of total degree $d+1$.
Here, $\op{ch}_i$ is treated as a degree $i$ polynomial variable, and plays the role of the universal $i$th Chern character.
The proof proceeds in two steps: (1) first we will show in section \ref{s:chern} that $\mathsf{cw}_d$ is injective when restricted to degree
$d+1$ polynomials, (2) we next use a spectral sequence induced from the "diagonal filtration" in section \ref{s:gf} to obtain
an upper bound for the second cohomology.

Our next result is a local, universal Grothendieck--Riemann--Roch theorem.
To state it, we start with a tensor $\lie{witt}_d$-module $\cV$.
Roughly, this is a dg module obtained as the global sections of some natural vector bundle on the formal punctured
disk.
Now, such a tensor module defines a homomorphism 
\begin{equation}\label{}
  \mathsf{L}_V \colon \lie{witt}_d \to \op{End}(\cV) . 
\end{equation}
On the other hand, $\op{End}(\cV)$ is a dg Tate vector space. 
It follows \cite{FHK} that its second Lie algebra cohomology is one-dimensional generated by the preferred class 
\begin{equation}\label{}
  [z^\1 \wedge \cdots \wedge z^{\dd} \wedge P] \in \HH^2_{Lie}(\op{End}(\cV)) 
\end{equation}
where $[P] \in H^{d-1}(\mathring{D}^d, \sO)$ is the $d$-dimensional residue class.
We will prove in section \ref{s:GRR} that
\begin{equation}\label{}
  \mathsf{L}_V^* [z^\1 \wedge \cdots \wedge z^{\dd} \wedge P] = \op{Td} \op{Ch}(\cV) \in \HH^2_{Lie}(\lie{witt}_d) .
\end{equation}

\paragraph
Most of the essential definitions and constructions used in this paper have appeared in \cite{GWvir2d} where we proved versions of the main
results of this paper in the case $d=2$.
We recall the neccesary notations in \ref{s:background} and \ref{s:chern}, but we refer the reader to \textit{loc.
cit.} for more details.

\paragraph
In the main text we also study an extension of dg Lie algebra $\lie{witt}_d$ by the dg Lie algebra of $\lie{gl}
_r$-valued currents on the punctured $d$-disk.
This is $\lie{gl}_r \otimes \cJ_d$ where $\cJ_d \simeq \R \Gamma(\mathring{D}^d, \sO)$ is the Jouanolou model for the
structure sheaf.
On its own, central extensions of this current dg Lie algebra were introduced as higher-dimensional
Kac--Moody algebras in \cite{FHK}.
Our work extends the classification of central exstensions for the semi-direct procut $\lie{witt}_d \ltimes \lie{gl}
_r (\cJ_d)$.
We also formulate and prove an extension of the Grothendieck--Riemann--Roch theorem to this setting.

\subsection*{Acknowledgements}
We thank Kevin Costello, Mikhail Kapranov, and Matt Szczesny for valuable discussions, particularly during the early stages of this work.

\section{Background on Jouanolou models}\label{s:background}
\subsection{Definitions and preliminary results}

Refer to \cite[\S 1.2]{FHK}, \cite[\S 1]{GWWchiral}, and specifically \cite[\S 1]{GWvir2d} for more details.

\paragraph
Let $\cJ_d^{poly}$ denote the Jouanlou model for the structure sheaf on $\AA^d - \{0\}$.
It is a dg algebra defined over $\C[z] = \C[z^\1, \ldots, z^\dd]$.
As a graded ring it is the following quotient of a graded polynomial ring
\begin{equation}\label{}
  \frac{\C[z^\s,x^\s,\d x^\s]_{\1 \leq \s \leq \dd}}{\<\sum_{\s=\1}^d z^\s x^\s -1, \quad \sum_{\s=\1}^d z^\s \d x^\s \> }
\end{equation}
where $z^\s, x^\s$ carry degree zero, and $\d x^\s$ carries degree $+1$ for each $\s$.
Note that $\cJ_d^{poly}$ is concentrated in degrees $0,\ldots, d-1$.
The differential is defined by
\[
  \dbar z^\s = 0, \quad \dbar x^\s = \d x^\s .
\]
Let $\cJ_d$ be its completion over the $z$-variables.
The latter is a dg algebra defined over $\C\llbracket z \rrbracket = \C\llbracket z^\1, \ldots, z^\dd \rrbracket$.

\paragraph
The Jouanolou model $\cJ_d$ (and $\cJ_d^{poly}$) is a $GL_d$-equivariant dg algebra meaning the product and
differential are $GL_d$
equivariant.
Let $\lie{q}$ denote the $d$-dimensional defining representation of $GL_d$.

\begin{proposition}[\cite{FHK}]
\label{prop:Jouanolou}
The cohomology of $\cJ^{poly}_d$ is concentrated in degrees zero and $d-1$.
There are natural $GL_d$-equivariant isomorphisms
\begin{equation}\label{}
  H^0(\cJ_d^{poly}) \cong \op{S}^\bu (\lie{q}^*) = \C[z]
\end{equation}
and
\begin{equation}\label{}
  H^{d-1} (\cJ_d^{poly}) \cong \op{S}^\bu (\lie{q}) \otimes \op{det}(\lie{q}) .
\end{equation}
Similar results hold for $\cJ_d$ where the first isomorphism is replaced by
\[
  H^0(\cJ_d) \cong \Hat{\op{S}}(\lie{q}^*) = \C\llbracket z \rrbracket .
\]
\end{proposition}

Let
\begin{equation}\label{}
  P \define \sum_{\s=1}^d (-1)^{\s-1} x^{\s} \d x^\1 \cdots \Hat{\d x^\s} \cdots \d x^\dd
\end{equation}
where the hat means that we exclude that wedge factor.
This is an element of $\cJ_d$ of degree $(d-1)$ and transforms in the determinant representation of $GL_d$.
For $\bk = (k_\1, \ldots, k_\dd)$ a $d$-tuple of non-negative integers,
\[
\del^{(\bk)} P \define \frac{(\sum_\ii k_\ii +1)!}{\prod_\ii k_\ii} (x^\1)^{k_\1} \cdots (x^\dd)^{k_\dd}  P .
\]
The collection $\{[\del^{(\bk)} P]_{\dbar}\}$ where $\bk$ ranges over all $d$-tuples of non-negative integers is a basis for
$H^{d-1}(\cJ_d)$.

\paragraph
The Jouanolou model for the tangent sheaf $\sT$ on $\AA^{d} - \{0\}$ is the polynomial dg Witt algebra denoted
$\lie{witt}_d^{poly}$.
It is naturally a dg Lie algebra whose differential we will also denote $\dbar$.
It is freely generated over $\cJ_d^{poly}$ by generators we denote $\del_\1, \ldots, \del_\dd$.
Equivalently, as a
cochain complex
\begin{equation}\label{}
  \lie{witt}^{poly}_d = \cJ_d^{poly} \otimes \lie{q}
\end{equation}
where $\lie{q} = \op{span}\{\del_\1, \ldots, \del_{\dd}\}$.

There is a natural $\cJ_d$-linear quasi-isomorphism of dg Lie algebras
\begin{equation}\label{}
  \lie{witt}_d^{poly} \xto{\simeq} \op{Der} \cJ^{poly}_d
\end{equation}
which is defined by
\[
  \del_{\ii} z^\jj = \delta_{\ii}^{\jj}, \quad \del_{\ii} x^{\jj} = -x^\ii x^\jj, \quad \del_{\ii} \d x^{\jj} = -
  x^{\jj} \d x^{\ii} - x^{\ii} \d x^{\jj} .
\]
As is usual, if $X \in\lie{witt}_d^{poly}$ then we denote $X f$ its action on $f \in \cJ_d$.
Similar remarks and formulas are true for the formal completion over the $z$-variables denoted $\lie{witt}_d$.

\begin{proposition}[\cite{GWvir2d}]
\label{prop:JouanolouWitt}
The cohomology of $\lie{witt}^{poly}_d$ is concentrated in degrees zero and $d-1$.
There is a natural $GL_d$-equivariant isomorphism of Lie algebras
\begin{equation}\label{}
  H^0(\lie{witt}_d^{poly}) \cong \op{S}^\bu (\lie{q}^*) \otimes \lie{q} = \lie{w}_d^{poly}
\end{equation}
and
\begin{equation}\label{}
  H^{d-1} (\lie{witt}_d^{poly}) \cong \op{S}^\bu (\lie{q}) \otimes \op{det}(\lie{q}) \otimes \lie{q}.
\end{equation}
Similar results hold for $\cJ_d$ where the first isomorphism is replaced by
\[
  H^0(\lie{witt}_d) \cong \lie{w}_d .
\]
\end{proposition}

\paragraph{Tensor modules}\label{TensorMod}
Let $V$ be a finite-dimensional $\lie{gl}_d$-module.
Define the cochain complex
\begin{equation}\label{}
  \cV \define \cJ_d \otimes V .
\end{equation}
We extend the representation $\rho_V \colon \lie{gl}_d \to \op{End}(V)$ in a $\cJ_d$-linear way which we denote by the
same symbol
\begin{equation}\label{}
  \rho_V \colon \cJ_d \otimes \lie{gl}_d \to \op{End}(\cV) .
\end{equation}
For $X \in \lie{witt}_d$, define $\mathsf{L}_{V}(X) \in \op{End}(\cV)$ by the formula
\begin{equation}\label{}
  \mathsf{L}_V(X) (f \otimes v) \define (Xf) \otimes v + \rho_V(JX) (f \otimes v)
\end{equation}
where $JX \in \cJ_d \otimes \lie{gl}_d$ is the Jacobian associated to $X$.
If $X = g^\ii \del_{\ii}$ it is defined by
\begin{equation}\label{eq:jacobian}
  (J X )_{\jj}^\ii = \frac{\del g^\ii}{\del z^\jj} \in \cJ_d .
\end{equation}
It is immediate to verify that $X \mapsto \sfL_X$ endows $\cV$ with the structure of a dg $\lie{witt}_d$-module.

We note that $H^0 (\lie{\cV})$ is naturally a \textit{tensor} $H^0(\lie{witt}_d) = \lie{w}_d$-module which is obtained
by restricting $V$ along
\begin{equation}\label{}
  \{X \in \lie{w}_d\; | \; X(0) = 0\} = \lie{w}_d^{\geq 0} \twoheadrightarrow \lie{gl}_d
\end{equation}
and coinducing along $\lie{w}_d^{\geq 0} \hookrightarrow \lie{w}_d$.
For this reason, we refer to such $\cV$ as a tensor module for $\lie{witt}_d$.
These modules appear in our formulation of the universal Grothendieck--Riemann--Roch theorem in section \ref{s:GRR}.

\paragraph{Matrix extension}
For a positive integer $r$ we will consider the dg algebra of $\cJ_d$-valued $r \times r$ matrices $\lie{gl}
_r(\cJ_d)$.
Its differential is induced from $\dbar$ and the product is defined using the commutative multiplication on $\cJ_d$
together with ordinary matrix multiplication.
Via commutator, we may treat it as a dg Lie algebra.
Invariant polynomials on $\lie{gl}_r$ of degree $d+1$ define central extensions of $\lie{gl}_r(\cJ_d)$ called
higher-dimensional Kac--Moody algebras \cite{FHK}.
We are ultimately interested in central extensions of this dg Lie algebra extended by the dg Witt algebra.

Entry-wise, $\lie{witt}_d$ acts on $\lie{gl}_r (\cJ_d)$ by derivations.
We therefore can form the semi-direct product dg Lie algebra
\begin{equation}\label{}
  \lie{witt}_d \ltimes \lie{gl}_r(\cJ_d) .
\end{equation}
In sections \ref{s:chern} and \ref{s:gf} we provide a classification of central extensions of this dg Lie algebra.
This dg Lie algebra is also the main ingredient in our formulation of the universal Grothendieck--Riemann--Roch theorem
in section \ref{s:GRR}.

\subsection{Flatness of the Jouanolou model}
We will use the following algebraic result in the main part of the paper.

\begin{proposition}\label{prop:flatness}
The Jouanolou model $\cJ_d$ is flat over $\C\llbracket z\rrbracket=\C\llbracket z^{\1},\dots,z^{\dd}\rrbracket$. An analogous statement is valid in the polynomial setting.
\end{proposition}

\begin{proof}
  The proof of \cite{FGY} can be modified to apply to the Jouanolou model in this setting. For the reader’s convenience, we present a more direct argument.

  As above, we denote $\C\llbracket z \rrbracket = \C\llbracket z^\1, \ldots, z^\dd \rrbracket$ and $\C[x] = \C[x^\1,
\ldots,x^\dd]$.
Consider the degree-zero component
\[
\cJ_d^0=\frac{\C\llbracket z\rrbracket[x]}{\<\sum_{\s=1}^d z^{\s}x^{\s}
-1\>}
\]
which is naturally an algebra over $\C\llbracket z \rrbracket$. We will show that $\cJ_d^0$ is flat as a $\C\llbracket z\rrbracket$-algebra, and that $\cJ_d$ is flat over $\cJ_d^0$.

Let $\mathbf{p}$ be a prime ideal in $\cJ_d^0$, and denote $\mathbf{q}=\mathbf{p}\cap \C\llbracket z \rrbracket$. Since $\sum_{\s=1}^d z^{\s}x^{\s}=1$ in $\cJ_d^0$, there exists $\ii$ such that $z^{\ii}\notin \mathbf{q}$. Consequently, $z^{\ii}$ becomes invertible in the localizations $\C\llbracket z \rrbracket_{\mathbf{q}}$ and $(\cJ_d^0)_{\mathbf{p}}$. Observe that
\[
\cJ_d^0\big[\tfrac{1}{z^{\ii}}\big]\simeq
\C\llbracket z \rrbracket\big[\tfrac{1}{z^{\ii}}\big][x^{\1},\dots,\widehat{x^{\ii}},\dots,x^{\dd}],
\]
where $\widehat{x^{\ii}}$ indicates that the variable $x^{\ii}$ is omitted. Localizing further at $\mathbf{p}$, we obtain
\[
(\cJ_d^0)_{\mathbf{p}}\simeq
\big(\C\llbracket z \rrbracket_{\mathbf{q}}[x^{\1},\dots,\widehat{x^{\ii}},\dots,x^{\dd}]\big)_{\mathbf{p}'},
\]
where $\mathbf{p}'=(\C\llbracket z \rrbracket-\mathbf{q})^{-1}\mathbf{p}$. Hence $(\cJ_d^0)_{\mathbf{p}}$ is flat over $\C\llbracket z \rrbracket_{\mathbf{q}}$ for every prime $\mathbf{p}\subset \cJ_d^0$. It follows that $\cJ_d^0$ is flat over $\C\llbracket z \rrbracket$.

Observe that we have the following short exact sequence of $\cJ^0_d$-modules:
\[
0\rightarrow \cJ^0_d\xrightarrow{\iota} (\cJ^0_d)^{\oplus d}\rightarrow \cJ^1_d\rightarrow 0,
\]
where $\iota(f)=f\bigl(\sum_{\s=1}^d z^{\s}\bar{\partial}x^{\s}\bigr)$. Consider the map $r\colon (\cJ^0_d)^{\oplus d}\rightarrow \cJ^0_d$ defined by
\[
r(f_1,\dots,f_d)=\sum_{\s=1}^d f_{\s}x^{\s}.
\]
Then $r\circ \iota=\mathrm{id}$, which shows that $\cJ^1_d$ is a direct summand of $(\cJ^0_d)^{\oplus d}$. Any
direct summand of a free $\cJ^0_d$-module is projective, hence in particular flat. Consequently, $\cJ_d=\wedge^\bu_{\cJ^0}\cJ^1_d$ is flat over $\cJ^0_d$, and therefore also flat over $\C\llbracket z \rrbracket$.
\end{proof}

\section{Universal Chern--Weil homomorphism}\label{s:chern}

In this section we recall the construction of a homomorphism 
\begin{equation}\label{eq:cw1}
  \mathsf{cw}_d \colon \C[\op{ch}_1,\ldots,\op{ch}_d]_{d+1} \to \HH^2_{Lie} (\lie{witt}_d) .
\end{equation}
from \cite[section 3]{GWvir2d}.
Here, we treat $\C[\op{ch}_1,\ldots,\op{ch}_d]$ as a graded polynomial ring with generator $\op{ch}_k$ carrying weight
$k$.
In particular, the dimension of its weight $d+1$ subspace is one less than the partition number of the integer $d+1$:
\begin{equation}\label{}
  \dim \C[\op{ch}_1,\ldots,\op{ch}_d]_{d+1} = p(d+1) - 1 
\end{equation}
since we through away the single block partition as we do not include $\op{ch}_{d+1}$.

After recalling the definition of $\mathsf{cw}_d$ we move onto the first main result in this paper which is to show that
$\mathsf{cw}_d$ is injective for all $d$.
The fact that $\op{cw}_1$ is an isomorphism is a classic result in conformal field theory and vertex algebras.
Our proof strategy in this section for general dimension $d$ is a generalization of the argument in \cite{GWvir2d}
for $d=2$.

In the last part of the section we generalize the Chern--Weil homomorphism \ref{eq:cw1} with $\lie{witt}_d$ replaced by
the semi-direct product $\lie{witt}_d \ltimes \lie{gl}_r(\cJ_d)$.

\subsection{The universal Chern--Weil homomorphism}
In this section we will be rather brief and refer to \cite{GWvir2d} for more details and proofs.

\paragraph
There is a universal element 
\[
\lie{c}\in C^\bu_{Lie}\left(\lie{witt}_d ; \br C_\bu^\lambda(\cJ_d) \right)^1
\]
in the Chevalley--Eilenberg cochain complex of $\lie{witt}_d$ with
coefficients in the Connes' reduced cyclic complex computing reduced cyclic homology of $\cJ_d$.
This element is of total degree $+1$ and decomposes as
\begin{equation}\label{eq:c}
  \lie{c} = \lie{c}_1 + \lie{c}_2+ \cdots 
\end{equation}
where
\begin{equation}\label{}
  \lie{c}_{k+1} \colon \wedge^{k+1} \lie{witt}_d \to \br C_\bu^\lambda(\cJ_d)
\end{equation}
is defined by
\begin{align}
  \lie{c}_{k+1}(T_0\wedge\cdots\wedge T_{k}) &=\frac{(-1)^k}{2^k} \sum_{\permute \in S_k}(\pm)^{'} \cdot \op{Tr}\left(J T_{0}, J T_{\permute(1)}, \ldots, J T_{\permute(k)}\right) \\
                                             & =\frac{(-1)^k}{2^k\cdot(k+1)} \sum_{\permute' \in S_{k+1}}(\pm)^{''}
                                             \cdot \op{Tr}\left(J T_{\permute'(0)}, J T_{\permute'(1)}, \ldots, J T_{\permute'(k)}\right) 
\end{align}
where
$$
(\pm)'=\mathrm{sgn}(\permute)\cdot\varepsilon(\permute;T_1,\ldots,T_k) ,
$$
$$
(\pm)''=\mathrm{sgn}(\permute')\cdot\varepsilon(\permute';T_0,T_1,\ldots,T_k) ,
$$
and $J \colon \lie{witt}_d \to \lie{gl}_d(\cJ_d)$ is defined in equation \eqref{eq:jacobian}. 

For $A$ a commutative dg algebra, Connes' complex $\overline{C}_\bu^\lambda(A)$ admits, itself, the natural structure of a
commutative dg algebra \cite[proposition 3.7]{loday1984cyclic}. 
It is defined by the formula
\begin{equation}\label{eqn:starproduct}
x*y \define x\times(By),
\end{equation}
where
\[
B(a_0,a_1,\dots,a_d)=\sum^n_{i=0}\pm (1,a_i,\dots,a_n,a_0,\dots,a_{i-1}),
\]
and 
\[
\pm=(-1)^{i\cdot n}\cdot\varepsilon(\permute_i;a_0,a_1,\dots,a_d),\quad \permute_i(a_0,a_1,\dots,a_d)=(a_i,\dots,a_n,a_0,\dots,a_{i-1})
\]
The operation $\times$ is the so-called shuffle product, given by 
\[
(a,a_1,\dots,a_p)\times (a',a_{p+1},\dots,a_{p+q})=\sum_{\text{Shuffle}}(\pm)\cdot\mathrm{sgn}(\permute)\cdot (aa',a_{\permute^{-1}(1)},\dots,a_{\permute^{-1}(p+q)}),
\]
\[
(\pm)=(-1)^{|a'|\sum^p\limits_{i=1}|a_i|}\cdot\varepsilon(\permute_i;a_1,\dots,a_{p+q})
\]
where the sum is over all permutations $\permute$ of $\{1,\dots,p+q\}$ such that $\permute(1)<\cdots<\permute(p)$ and $\permute(p+1)<\cdots<\permute(p+q)$.

Define the cyclic cocycle~$\rho \in C^\bu_\lambda(\cJ_d)$ of total degree one by the formula
\begin{equation}\label{eq:rho}
  \rho(f_0,\ldots,f_d) = \op{Res}(f_0 \del f_1 \cdots \del f_d) .
\end{equation}
where $\op{Res}(- \d^d z)$ is the unique $GL_d$-invariant degree $-d$ cochain map $\cJ_d \to \C$ with the property that $\op{Res}
(\d^d z) =1$.

\paragraph
We are now ready to define the universal Chern--Weil homomorphism
\begin{equation}\label{}
  \mathsf{cw}_d \colon \C[y_1,\ldots,y_d]_{d+1} \to \HH^2_{Lie}(\lie{witt}_d) ,
\end{equation}
where the $y_1,\ldots,y_d$ are polynomial variables and $y_i$ carries homogenous degree $i$.
The subscript in the domain means that we look at total degree $d+1$ elements.
Let $i_1,\ldots,i_d$ be integers such that $i_1+2i_2+\cdots+d i_d=d+1$.
Define
\begin{equation}\label{eq:chernwitt}
  \mathsf{cw}_d (y_1^{i_1} y_2^{i_2} \cdots y_d^{i_d}) \define \rho \left(\lie{c}_1^{*i_1} * \lie{c}_2^{*i_2} * \cdots
    * \lie{c}_d^{i_d}\right) .
\end{equation}
Implicit in this definition is the claim that the this element is in fact a cocycle and we are taking its cohomology
class.

\paragraph[p:cwdr]
The above construction naturally generalizes to the algebra $\lie{witt}_d\ltimes \lie{gl}_r(\cJ_d)$ in the following way.
Suppose that 
\[
  \theta \in \op{S}^{k+1}(\lie{gl}_r^*)^{\lie{gl}_r}
\]
is an invariant polynomial on $r \times r$ matrices of
homogeneous polynomial degree $k+1$.
Define
\begin{equation}\label{}
  \gamma_{\theta} \colon \wedge^{k+1} \lie{gl}_r(\cJ_d) \to \br C_\bu^\lambda(\cJ_d)[1]
\end{equation}
by the formula
\begin{align}
  \gamma_{\theta}(M_0\otimes f_0\wedge\cdots\wedge M_k\otimes f_k) &=\frac{(-1)^k}{2^k} \sum_{\permute \in S_k}(\pm)^{'}
  \cdot \theta(M_0,M_{\permute(1)},\dots,M_{\permute(k)})\cdot \left(f_0, f_{\permute(1)}, \ldots, f_{\permute(k)}\right) .
\end{align}
It follows immediately that $\gamma_\theta$ is a cochain map, $\lie{gl}_r$-equivariant, and also $\lie{witt}
_d$-equivariant.

\paragraph
Let $\theta$ be a degree $k$ invariant polynomial on $r \times r$ matrices as above and suppose $i_1,\ldots,i_{d-k}$ are integers such that 
\begin{equation}\label{}
  i_1 + 2 i_2 + \cdots + (d-k) i_{d-k} = d-k . 
\end{equation}
Define
\begin{equation}\label{}
  \mathsf{cw}_{d,r} (y_1^{i_1} \cdots y_{d-k}^{i_{d-k}} \theta) \define \mathsf{cw}\left( \lie{c}_1^{*i_1} * \cdots
  *\lie{c}_{d-k}^{*i_{d-k}} * \gamma_\theta \right) ,
\end{equation}
which is an element in $\HH^2_{Lie}(\lie{witt}_d \ltimes \lie{gl}_r(\cJ_d))$.

Consider now the ring 
\begin{equation}\label{}
  \C[y_1,\ldots,y_d] \otimes \op{S}^\bu(\lie{gl}_r^*)^{\lie{gl}_r} \cong \C[y_1,\ldots,y_d, \lie{t}_1,\ldots, \lie{t}_r] 
\end{equation}
which we assign a non-cohomological weight grading by declaring that $y_i$ is degree $i$ and $\lie{t}_{j} \in \op{S}^j(\lie{gl}^*_r)$
defined by 
\begin{equation}\label{}
  \lie{t}_j (A) = \op{tr} (A^j)
\end{equation}
is of degree $j$.
Our construction defines a homomorphism
\begin{equation}\label{}
  \mathsf{cw}_{d,r} \colon \C[y_1,\ldots,y_d, \lie{t}_1,\ldots,\lie{t}_r]_{d+1} \to \HH^2_{Lie}(\lie{witt}_d \ltimes
  \lie{gl}_r(\cJ_d)) 
\end{equation}
Explicitly, on the trace generators, this is defined as
\begin{equation}\label{eq:chernwittMixed}
  \mathsf{cw}_{d,r} (y_1^{i_1} y_2^{i_2} \cdots y_d^{i_d} \cdot \lie{t}_1^{k_1} \cdots \lie{t}_r^{k_r} ) \define \rho \left(\lie{c}_1^{*i_1} * \lie{c}_2^{*i_2} * \cdots
  * \lie{c}_d^{*i_d}* \gamma_{\lie{t}_1}^{* k_1} * \cdots \gamma_{\lie{t}_r}^{* k_r} \right) 
\end{equation}
where $i_1,\ldots,i_d,k_1,\ldots,k_r$ are integers satisfying the relation 
\begin{equation}\label{}
  i_1+2i_2+\cdots+d i_d+k_1 + 2 k_2 + \cdots + r k_r =d+1 .
\end{equation}

\subsection{Injectivity}

We will show (a) the cohomology class of any cocycle of the form \eqref{eq:chernwitt} is nontrivial
and (b) such classes are linearly independent.
In other words, we prove the map $\mathsf{cw}_d$
is injective.
Of course, this statement is obvious and well-known in the case $d=1$.
In \cite{GWvir2d} have proven injectivity when $d=2$.
We treat the general case here.

Our notation in this section is 
\begin{equation}\label{}
  \op{ch}_1^{i_1} \cdots \op{ch}_d^{i_d} \define \mathsf{cw}_d (y_1^{i_1} \cdots y_d^{i_d}) 
\end{equation}
where $\sum^{d}\limits_{p=2} p i_p = d+1$.

\paragraph
We begin with the following lemma which expresses the top trace class as a linear combination of elements in the image
of $\mathsf{cw}_d$.

\begin{lemma}\label{lem:chd+1}
In $\mathbb{H}^2_{\mathrm{Lie}}\left(\lie{witt}_{d}\right)$, the class 
\[
  \mathrm{ch}_{d+1} \define \rho(\lie{c}_{d+1}) 
\]
can be expressed as linear combination of elements of the form 
\[
  (\mathrm{ch}_1)^{i_1}(\mathrm{ch}_2)^{i_2}\cdots (\mathrm{ch}_d)^{i_d}
\]
where $\sum^{d}\limits_{p=1}p i_p=d+1$.
In other words, $\op{ch}_{d+1}$ is in the image of $\mathsf{cw}_d$.
\end{lemma}
\begin{proof}
    The proof is by the Procesi-Razmyslov fundamental trace identity, equivalently the multilinear polarized
    Cayley-Hamilton trace identity (see for example \cite[Theorem 4.3(b)]{PROCESI1976306}). We apply the following trace identity to the matrix ring $\lie{gl}_d(\cJ_d^{\otimes d+1})$
    \[
    \sum_{\tau\in S_{d+1}}\mathrm{sgn}(\tau)\prod_{C=(i_1...i+r)\in \mathrm{Cycles}(\tau)}\mathrm{Tr}\big(M_{i_1}\cdots M_{i_r}\big)=0,\quad M_0,...,M_d\in \lie{gl}_d(\cJ_d^{\otimes d+1}).
    \]
    The explicit formula is 
    \[
      \op{ch}_{d+1} =
\frac{(-1)^{d+1}}{d!}
\op{exp}\left(
  \sum_{k=1}^d (-1)^{k-1}(k-1)!\op{ch}_k t^k
\right) |_{t^{d+1}} 
\]
where the subscript means we take the coefficient of $t^{d+1}$.
\end{proof}

\paragraph
In what follows it will be more convenient to label our degree $d+1$ classes by partitions of the integer $d+1$.
For example 
\[
  \op{ch}_{(1,\ldots,1)} = \op{ch}_1^{d+1} 
\]
while
\[
  \op{ch}_{(2,1,\ldots,1)} = \op{ch}_1^{d-1} \op{ch}_2 . 
\]
More generally, if $\lambda = (1^{i_1}, \ldots, d^{i_d})$ is a partition of $d+1$ then we define 
\begin{equation}\label{}
  \op{ch}_\lambda \define \op{ch}_1^{i_1} \cdots \op{ch}_d^{i_d}
\end{equation}
We also set $\op{ch}_{(d+1)} = \op{ch}_{d+1}$.

\paragraph{Injectivity in dimension $d=2$}
Before diving into the main proof of injectivity, we survey the $d=2$ case. 
When $d=2$ we have two classes $\op{ch}_1^3$ and $\op{ch}_1 \op{ch}_2$ which span $\HH^2_{Lie}(\lie{witt}_d)$.
Our strategy is to find, for example, a chain $\til{\mathbb{v}}\in \bigwedge^{\bullet}\lie{witt}_{2}$ of total degree $-2$
which detects the nontriviality of $\op{ch}_1^3$ at the
cochain level:
$$
(\dbar + \bd^{(2)})\til{\mathbb{v}}=0,\quad (\mathrm{ch}_1^3)(\til{\mathbb{v}})\neq 0.
$$
Here, $\dbar$ is the linear part of the differential and $\bd^{(2)}$ is induced from the Lie bracket.
In other words, $\til{\mathbb{v}}$ is a cycle in the Chevalley--Eilenberg complex for $\lie{witt}_2$, and it pairs
nontrivially with the cocycle $\op{ch}_1^3$.
This shows that the image of $\mathsf{cw}_2$ is at least one-dimensional.
To show that it is two-dimensional, it suffices to find another chain which pairs nontrivially with a linearly
independent cohomology class.
One obvious choice is the class $\op{ch}_1 \op{ch}_2$.
Instead, we will find it easier to find a cycle which is dual to the class $\op{ch}_3$, which, in this dimension, is a linear combination 
\begin{equation}\label{}
  \op{ch}_3 = -\frac{1}{12} \op{ch}_1^3 + \frac12  \op{ch}_1 \op{ch}_2 .
\end{equation}

In dimension two, it is easy to write down the cycles which implement this strategy.
One starts with the two chains:
\begin{align*}
  \mathbb{v}_{(1,1,1)}(2) & \define  2 P_2\partial_{\1}\wedge (z^{\1})^3\partial_{\1}\wedge (z^{\2})^2\partial_{\2} \\
  \mathbb{v}_{(3)}(2) & \define P_2\partial_{\1}\wedge z^{\1}z^{\1}\mathsf{Eu}\wedge z^{\2} \mathsf{Eu}
  -P_2\partial_{\1} \wedge z^{\1}z^{\2} \mathsf{Eu}\wedge z^{\1} \mathsf{Eu},
\end{align*}
where $\mathsf{Eu} = z^\1 \del_{\1} + z^{\2} \del_{\2}$ is the total Euler vector field.
We will explain momentarily, see definition \ref{dfn:partitionv}, how one associates a chain of this type to a general
partition.
The key point is that these chains pair nontrivially with $\op{ch}_1^3$ and $\op{ch}_3$ respectively, see table
\ref{tbl:pairing2} (and ignore the column $\op{ch}_{(2,1)}(2)$ which we have included only for transparency).

\begin{table}[!h]
        \centering
        
\begin{tabular}{|p{0.10\textwidth}|p{0.12\textwidth}|p{0.10\textwidth}|p{0.10\textwidth}|}
\hline 
  & $\displaystyle \mathrm{ch}_{(1,1,1)}(2)$ & $\displaystyle \mathrm{ch}_{(2,1)}(2)$ & $\displaystyle \mathrm{ch}_{(3)}(2)$ \\
\hline 
 $\displaystyle \mathbb{v}_{(1,1,1)}( 2)$ & $\displaystyle \neq 0$ & $\displaystyle \neq 0$ & $\displaystyle 0$ \\
\hline 
 $\displaystyle \mathbb{v}_{(3)}( 2)$ & $\displaystyle \neq 0$ & $\displaystyle \neq 0$ & $\displaystyle \neq 0$ \\
 \hline
\end{tabular}
\caption{Pairing matrix, $d=2$}
\label{tbl:pairing2}
        \end{table}

On the other hand, the chains above are not cycles in the complex $C_\bu^{Lie}(\lie{witt}_d)$ computing the Lie
algebra hyperhomology of $\lie{witt}_d$.
Fortunately, we can run a version of descent to introduce deformations of these chains which are cocycles.
Explicitly, in dimension $d=2$, one choice of descent data leads to:
\begin{align*}
\til{\mathbb v}_{(1,1,1)}(2)
&=
2P_2\partial_{\1}\wedge (z^{\1})^3\partial_{\1}\wedge (z^{\2})^2\partial_{\2}
\\
&\quad
+
\left(
7z^{\1}x^{\2}
+
  (z^{\1})^2x^{\1}x^{\2}
\right)
\partial_{\1}\wedge (z^{\2})^2\partial_{\2}
\\
&\quad
+
\left(
x^{\1}
+
z^{\2}x^{\1}x^{\2}
\right)
\partial_{\1}\wedge (z^{\1})^3\partial_{\1}
\\
&\quad
-
\left(
3(z^{\1})^2(x^{\1})^2
-
6z^{\1}x^{\1}
+
7
\right)
z^{\1}x^{\2}\,
\partial_{\1}\wedge z^{\2}\partial_{\2}.
\end{align*}
and
\begin{align*}
  \til{\mathbb v}_{(3)}(2)
&=
P_2\partial_{\1}\wedge (z^{\1})^2\mathsf{Eu}\wedge z^{\2}\mathsf{Eu}
-
P_2\partial_{\1}\wedge z^{\1}z^{\2}\mathsf{Eu}\wedge z^{\1}\mathsf{Eu}
\\
&\quad
+
\left(4z^{\1}x^{\2}\partial_{\1}-2z^{\1}x^{\1}\partial_{\2}\right)\wedge z^{\2}\mathsf{Eu}
\\
&\quad
+
2x^{\1}\partial_{\1}\wedge (z^{\1})^2\mathsf{Eu}
\\
&\quad
+
\left(
-z^{\2}x^{\2}\partial_{\1}
+
2z^{\1}x^{\1}\partial_{\1}
+
z^{\2}x^{\1}\partial_{\2}
\right)
\wedge z^{\1}\mathsf{Eu}
\\
&\quad
+
\left(3x^{\2}\partial_{\1}-x^{\1}\partial_{\2}\right)\wedge z^{\1}z^{\2}\mathsf{Eu}
\\
&\quad
-
\left(
-2z^{\1}x^{\1}+11
\right)x^{\2}\,
\partial_{\1}\wedge z^{\1}\mathsf{Eu}
-
2x^{\2}\,
\partial_{\2}\wedge z^{\1}z^{\2}\mathsf{Eu}.
\end{align*}

The last remark is that these correction terms automatically pair trivially, by type reasons, with $\op{ch}_{(1,1,1)}$ and $\op{ch}_{(3)}
$.
Thus, the contents of table \ref{tbl:pairing2} are unchanged if we replace the $\mathbb{v}$'s with their cocycle
corrections $\til{\mathbb{v}}$'s.
        This completes the proof that $\mathsf{cw}_2$ is injective. 

\paragraph{Injectivity for $d\geq 2$}
Recall the Lie algebra of formal vector fields $\lie{w}_d$ on the $d$-disk is a dg Lie subalgebra of $\lie{witt}_d$.
Let $\lie{w}^{poly}_d \subset \lie{w}_d$ denote the Lie subalgebra of formal vector fields whose coefficients are
polynomial in $d$-variables.

We work with Chevalley--Eilenberg chains rather than cochains.
We let $\dbar + \bd^{(2)}$ denote the total differential acting on the Chevalley--Eilenberg complex
$C_\bu(\lie{witt}^{poly}_d)$ computing the Lie algebra homology of $\lie{witt}^{poly}_d$.
(So, $\dbar$ is the linear part of the differential and $\bd^{(2)}$ is induced from the bracket.)
In the following definition, we view $\wedge^\bu \lie{w}_d^{poly}$ as subspace of the total degree $(-d)$ part of
$C_\bu(\lie{witt}^{poly}_d)$.

\begin{definition}
  Let $\mathbf{Z}_1(d)\subset \wedge^d \lie{w}^{poly}_d$ be the subspace consisting of
    $$
    \mathbb{w}\in \wedge^d \lie{w}^{\mathrm{pol}}_d,\quad \mathbf{d}^{(2)}\mathbb{w}=0,\ \text{and}\  \mathbf{d}^{(2)}
    (P_d\partial_{\1}\wedge\mathbb{w})=\dbar(-).
    $$
    In other words, $d$-cochains $\mathbb{w}$ which are closed for the Chevalley--Eilenberg differential of formal vector fields and
    such that in the full Chevalley--Eilenberg complex for $\lie{witt}_d$ the quadratic part of the
    Chevalley--Eilenberg applied to $P_d \del_{\1} \wedge \mathbb{w}$ is $\dbar$-exact.
\end{definition}

%Recall the following classical fact in \cite[pp84,Lemma 2.]{FuksBook} (we reformulate in terms of homology by using the
%fact that the continuous linear dual of $H_{q}(\lie{w}_d^{poly})$ is isomorphic to $H^q(\lie{w}_d)$.
%\begin{lemma}
%  $H_{q}(\lie{w}_d^{poly})=0$ for $0<q\leq2n$.
%\end{lemma}
%This vanishing result easily leads to the following.
%Let $F_r C_\bu^{Lie}(\lie{witt}_d)$ be the increasing filtration which consists of chains which carry total graded
%polynomial degree $\leq r$.

\begin{lemma}\label{lem:descent}
  For $\mathbb{w}\in \mathbf{Z}_1(d)$, the class of the cochain 
  \[
    \mathbb{v}^{(d-1)} \define P_d\del_{\1} \wedge\mathbb{w}
  \]
  in $\op{gr}_{d+1} C^{Lie}_{\bu}(\lie{witt}_d)$ lifts
  to a cocycle in $F_{d+1} C^{Lie}_{\bu}(\lie{witt}_d)$.
\end{lemma}
\begin{proof}
    We refer to figure \ref{fig:Zig-Zag} to elucidate the following zig-zag argument.
    Starting with $\mathbb{v}^{(d-1)}$, we construct a total degree $(-2)$ cocycle in $C^{Lie}_\bu (\lie{witt}_d)$ of the form 
    \begin{equation}\label{eq:sumcocycle}
      \til{\mathbb{v}} = \mathbb{v}^{(d-1)} + \cdots + \mathbb{v}^{(0)} 
    \end{equation}
where $\mathbb{v}^{(j)}$ has internal (Joaunolou) degree $j$.
    By assumption, there is $\mathbb{v}^{(d-2)}$ such that 
    \[
      \dbar \mathbb{v}^{(d-2)} + \bd^{(2)} \mathbb{v}^{(d-1)} = 0 .
    \]
    From this equation, we see that $\dbar\bd^{(2)} \mathbb{v}^{(d-2)} = 0$.
    Since the cohomology of the Joaunolou model is concentrated in degrees up to $0$ and $d-1$ we see that there exists
    $\mathbb{v}^{(d-3)}$ such that $\dbar \mathbb{v}^{(d-3)} + \bd^{(2)} \mathbb{v}^{(d-2)} = 0$.
    Continuing on in this way, we produce $\mathbb{v}^{(d-1)}, \mathbb{v}^{(d-2)}, \ldots, \mathbb{v}^{(1)}, \mathbb{w}$
    together which
    satisfy the descent equations 
    \begin{align*}
      \dbar \mathbb{v}^{(d-1)} & = 0 \\
      \dbar \mathbb{v}^{(d-2)} + \bd^{(2)} \mathbb{v}^{(d-1)} & = 0 \\
                                                             & \vdots \\
      \dbar \mathbb{w} + \bd^{(2)} \mathbb{v}^{(1)} & = 0 .
    \end{align*}
    In particular:
    \begin{equation}\label{}
      \left(\dbar + \bd^{(2)}\right)\left(\mathbb{v}^{(d-1)} + \cdots + \mathbb{v}^{(1)} + \mathbb{w}\right) = \bd^{(2)} 
      \mathbb{w} .
    \end{equation}
    In other words, the obstruction for the sum inside parentheses to be a cocycle is $\bd^{(2)} \mathbb{w}$.
    The issue is that we cannot assume that $\mathbb{w}$ is an element of $\wedge^2 \lie{w}_d^{poly}$, it is an element
    of $\wedge^2 (\lie{witt}_d^{poly})^0$.
    But, from the last descent equation we see that $\bd^{(2)} \mathbb{w}$ is $\dbar$-closed hence 
    \begin{equation}\label{}
      \bd^{(2)} \mathbb{w} \in \lie{w}^{poly}_d . 
    \end{equation}
    Since this element is automatically $\bd^{(2)}$-closed (where, we now mean the Chevalley--Eilenberg differential for
    $\lie{w}_d^{poly}$), it follows from $H_1^{Lie}(\lie{w}^{poly}_d) = 0$
    that there exists $\mathbb{u} \in \wedge^2 \lie{w}_d$ with the property that $\bd^{(2)} \mathbb{u} =
    \bd^{(2)} \mathbb{w}$.
    Finally, define:
    \begin{equation}\label{}
      \mathbb{v}^{(0)} \define \mathbb{w} - \mathbb{u} .
    \end{equation}
    The desired cocycle is then \eqref{eq:sumcocycle}.
    \begin{figure}[htbp]
        \centering

\tikzset{every picture/.style={line width=0.75pt}} %set default line width to 0.75pt        

\begin{tikzpicture}[x=0.75pt,y=0.75pt,yscale=-1,xscale=1]
%uncomment if require: \path (0,362); %set diagram left start at 0, and has height of 362

%Shape: Axis 2D [id:dp4320338525459244] 
\draw  (475.2,323.56) -- (474.49,35.63)(125.93,295.62) -- (513.92,294.67) (469.51,42.64) -- (474.49,35.63) -- (479.51,42.62) (132.94,300.6) -- (125.93,295.62) -- (132.92,290.6) (470,243.78) -- (480,243.76)(469.88,192.78) -- (479.88,192.76)(469.75,141.78) -- (479.75,141.76)(469.63,90.78) -- (479.63,90.76)(424.14,299.89) -- (424.11,289.89)(373.14,300.02) -- (373.11,290.02)(322.14,300.14) -- (322.11,290.14)(271.14,300.27) -- (271.11,290.27)(220.14,300.39) -- (220.11,290.39)(169.14,300.52) -- (169.11,290.52) ;
\draw   ;
%Straight Lines [id:da2900421336009549] 
\draw [color={rgb, 255:red, 189; green, 16; blue, 224 }  ,draw opacity=1 ][fill={rgb, 255:red, 74; green, 144; blue, 226 }  ,fill opacity=1 ]   (221,93) ;
\draw [shift={(221,93)}, rotate = 0] [color={rgb, 255:red, 189; green, 16; blue, 224 }  ,draw opacity=1 ][fill={rgb, 255:red, 189; green, 16; blue, 224 }  ,fill opacity=1 ][line width=0.75]      (0, 0) circle [x radius= 3.35, y radius= 3.35]   ;
%Straight Lines [id:da5500088828853273] 
\draw [color={rgb, 255:red, 0; green, 0; blue, 0 }  ,draw opacity=1 ][fill={rgb, 255:red, 74; green, 144; blue, 226 }  ,fill opacity=1 ]   (208,92.91) -- (169,92.91) ;
\draw [shift={(169,92.91)}, rotate = 180] [color={rgb, 255:red, 0; green, 0; blue, 0 }  ,draw opacity=1 ][fill={rgb, 255:red, 0; green, 0; blue, 0 }  ,fill opacity=1 ][line width=0.75]      (0, 0) circle [x radius= 3.35, y radius= 3.35]   ;
\draw [shift={(210,92.91)}, rotate = 180] [fill={rgb, 255:red, 0; green, 0; blue, 0 }  ,fill opacity=1 ][line width=0.08]  [draw opacity=0] (12,-3) -- (0,0) -- (12,3) -- cycle    ;
%Straight Lines [id:da1503395154909084] 
\draw [color={rgb, 255:red, 0; green, 0; blue, 0 }  ,draw opacity=1 ][fill={rgb, 255:red, 74; green, 144; blue, 226 }  ,fill opacity=1 ]   (221.56,102) -- (221.27,141) ;
\draw [shift={(221.27,141)}, rotate = 90.44] [color={rgb, 255:red, 0; green, 0; blue, 0 }  ,draw opacity=1 ][fill={rgb, 255:red, 0; green, 0; blue, 0 }  ,fill opacity=1 ][line width=0.75]      (0, 0) circle [x radius= 3.35, y radius= 3.35]   ;
\draw [shift={(221.58,100)}, rotate = 90.44] [fill={rgb, 255:red, 0; green, 0; blue, 0 }  ,fill opacity=1 ][line width=0.08]  [draw opacity=0] (12,-3) -- (0,0) -- (12,3) -- cycle    ;
%Straight Lines [id:da014870988023802267] 
\draw [color={rgb, 255:red, 189; green, 16; blue, 224 }  ,draw opacity=1 ][fill={rgb, 255:red, 74; green, 144; blue, 226 }  ,fill opacity=1 ]   (273,141) ;
\draw [shift={(273,141)}, rotate = 0] [color={rgb, 255:red, 189; green, 16; blue, 224 }  ,draw opacity=1 ][fill={rgb, 255:red, 189; green, 16; blue, 224 }  ,fill opacity=1 ][line width=0.75]      (0, 0) circle [x radius= 3.35, y radius= 3.35]   ;
%Straight Lines [id:da2043361360051017] 
\draw [color={rgb, 255:red, 0; green, 0; blue, 0 }  ,draw opacity=1 ][fill={rgb, 255:red, 74; green, 144; blue, 226 }  ,fill opacity=1 ]   (260,140.91) -- (221,140.91) ;
\draw [shift={(221,140.91)}, rotate = 180] [color={rgb, 255:red, 0; green, 0; blue, 0 }  ,draw opacity=1 ][fill={rgb, 255:red, 0; green, 0; blue, 0 }  ,fill opacity=1 ][line width=0.75]      (0, 0) circle [x radius= 3.35, y radius= 3.35]   ;
\draw [shift={(262,140.91)}, rotate = 180] [fill={rgb, 255:red, 0; green, 0; blue, 0 }  ,fill opacity=1 ][line width=0.08]  [draw opacity=0] (12,-3) -- (0,0) -- (12,3) -- cycle    ;
%Straight Lines [id:da7612997244093023] 
\draw [color={rgb, 255:red, 0; green, 0; blue, 0 }  ,draw opacity=1 ][fill={rgb, 255:red, 74; green, 144; blue, 226 }  ,fill opacity=1 ]   (273.56,150) -- (273.27,189) ;
\draw [shift={(273.27,189)}, rotate = 90.44] [color={rgb, 255:red, 0; green, 0; blue, 0 }  ,draw opacity=1 ][fill={rgb, 255:red, 0; green, 0; blue, 0 }  ,fill opacity=1 ][line width=0.75]      (0, 0) circle [x radius= 3.35, y radius= 3.35]   ;
\draw [shift={(273.58,148)}, rotate = 90.44] [fill={rgb, 255:red, 0; green, 0; blue, 0 }  ,fill opacity=1 ][line width=0.08]  [draw opacity=0] (12,-3) -- (0,0) -- (12,3) -- cycle    ;
%Straight Lines [id:da18437113410659078] 
\draw [color={rgb, 255:red, 0; green, 0; blue, 0 }  ,draw opacity=1 ][fill={rgb, 255:red, 74; green, 144; blue, 226 }  ,fill opacity=1 ]   (313,188.91) -- (274,188.91) ;
\draw [shift={(274,188.91)}, rotate = 180] [color={rgb, 255:red, 0; green, 0; blue, 0 }  ,draw opacity=1 ][fill={rgb, 255:red, 0; green, 0; blue, 0 }  ,fill opacity=1 ][line width=0.75]      (0, 0) circle [x radius= 3.35, y radius= 3.35]   ;
\draw [shift={(315,188.91)}, rotate = 180] [fill={rgb, 255:red, 0; green, 0; blue, 0 }  ,fill opacity=1 ][line width=0.08]  [draw opacity=0] (12,-3) -- (0,0) -- (12,3) -- cycle    ;
%Straight Lines [id:da18159335931665033] 
\draw [color={rgb, 255:red, 189; green, 16; blue, 224 }  ,draw opacity=1 ][fill={rgb, 255:red, 74; green, 144; blue, 226 }  ,fill opacity=1 ]   (321,192) ;
\draw [shift={(321,192)}, rotate = 0] [color={rgb, 255:red, 189; green, 16; blue, 224 }  ,draw opacity=1 ][fill={rgb, 255:red, 189; green, 16; blue, 224 }  ,fill opacity=1 ][line width=0.75]      (0, 0) circle [x radius= 3.35, y radius= 3.35]   ;
%Straight Lines [id:da27425030392278593] 
\draw [color={rgb, 255:red, 0; green, 0; blue, 0 }  ,draw opacity=1 ][fill={rgb, 255:red, 74; green, 144; blue, 226 }  ,fill opacity=1 ]   (320.56,203) -- (320.27,242) ;
\draw [shift={(320.27,242)}, rotate = 90.44] [color={rgb, 255:red, 0; green, 0; blue, 0 }  ,draw opacity=1 ][fill={rgb, 255:red, 0; green, 0; blue, 0 }  ,fill opacity=1 ][line width=0.75]      (0, 0) circle [x radius= 3.35, y radius= 3.35]   ;
\draw [shift={(320.58,201)}, rotate = 90.44] [fill={rgb, 255:red, 0; green, 0; blue, 0 }  ,fill opacity=1 ][line width=0.08]  [draw opacity=0] (12,-3) -- (0,0) -- (12,3) -- cycle    ;
%Straight Lines [id:da5538174997058338] 
\draw [color={rgb, 255:red, 189; green, 16; blue, 224 }  ,draw opacity=1 ][fill={rgb, 255:red, 74; green, 144; blue, 226 }  ,fill opacity=1 ]   (372,242) ;
\draw [shift={(372,242)}, rotate = 0] [color={rgb, 255:red, 189; green, 16; blue, 224 }  ,draw opacity=1 ][fill={rgb, 255:red, 189; green, 16; blue, 224 }  ,fill opacity=1 ][line width=0.75]      (0, 0) circle [x radius= 3.35, y radius= 3.35]   ;
%Straight Lines [id:da9487501648592407] 
\draw [color={rgb, 255:red, 0; green, 0; blue, 0 }  ,draw opacity=1 ][fill={rgb, 255:red, 74; green, 144; blue, 226 }  ,fill opacity=1 ]   (359,241.91) -- (320,241.91) ;
\draw [shift={(320,241.91)}, rotate = 180] [color={rgb, 255:red, 0; green, 0; blue, 0 }  ,draw opacity=1 ][fill={rgb, 255:red, 0; green, 0; blue, 0 }  ,fill opacity=1 ][line width=0.75]      (0, 0) circle [x radius= 3.35, y radius= 3.35]   ;
\draw [shift={(361,241.91)}, rotate = 180] [fill={rgb, 255:red, 0; green, 0; blue, 0 }  ,fill opacity=1 ][line width=0.08]  [draw opacity=0] (12,-3) -- (0,0) -- (12,3) -- cycle    ;
%Straight Lines [id:da43947256340248986] 
\draw [color={rgb, 255:red, 0; green, 0; blue, 0 }  ,draw opacity=1 ][fill={rgb, 255:red, 74; green, 144; blue, 226 }  ,fill opacity=1 ]   (372.56,251) -- (372.27,290) ;
\draw [shift={(372.27,290)}, rotate = 90.44] [color={rgb, 255:red, 0; green, 0; blue, 0 }  ,draw opacity=1 ][fill={rgb, 255:red, 0; green, 0; blue, 0 }  ,fill opacity=1 ][line width=0.75]      (0, 0) circle [x radius= 3.35, y radius= 3.35]   ;
\draw [shift={(372.58,249)}, rotate = 90.44] [fill={rgb, 255:red, 0; green, 0; blue, 0 }  ,fill opacity=1 ][line width=0.08]  [draw opacity=0] (12,-3) -- (0,0) -- (12,3) -- cycle    ;
%Straight Lines [id:da0649212091209419] 
\draw [color={rgb, 255:red, 189; green, 16; blue, 224 }  ,draw opacity=1 ][fill={rgb, 255:red, 74; green, 144; blue, 226 }  ,fill opacity=1 ]   (424,290) ;
\draw [shift={(424,290)}, rotate = 0] [color={rgb, 255:red, 189; green, 16; blue, 224 }  ,draw opacity=1 ][fill={rgb, 255:red, 189; green, 16; blue, 224 }  ,fill opacity=1 ][line width=0.75]      (0, 0) circle [x radius= 3.35, y radius= 3.35]   ;
%Straight Lines [id:da2438265525064589] 
\draw [color={rgb, 255:red, 0; green, 0; blue, 0 }  ,draw opacity=1 ][fill={rgb, 255:red, 74; green, 144; blue, 226 }  ,fill opacity=1 ]   (412,289.91) -- (373,289.91) ;
\draw [shift={(373,289.91)}, rotate = 180] [color={rgb, 255:red, 0; green, 0; blue, 0 }  ,draw opacity=1 ][fill={rgb, 255:red, 0; green, 0; blue, 0 }  ,fill opacity=1 ][line width=0.75]      (0, 0) circle [x radius= 3.35, y radius= 3.35]   ;
\draw [shift={(414,289.91)}, rotate = 180] [fill={rgb, 255:red, 0; green, 0; blue, 0 }  ,fill opacity=1 ][line width=0.08]  [draw opacity=0] (12,-3) -- (0,0) -- (12,3) -- cycle    ;
%Straight Lines [id:da8248821882706661] 
\draw [color={rgb, 255:red, 0; green, 0; blue, 0 }  ,draw opacity=1 ][fill={rgb, 255:red, 74; green, 144; blue, 226 }  ,fill opacity=1 ]   (424.56,251) -- (424.27,290) ;
\draw [shift={(424.27,290)}, rotate = 90.44] [color={rgb, 255:red, 0; green, 0; blue, 0 }  ,draw opacity=1 ][fill={rgb, 255:red, 0; green, 0; blue, 0 }  ,fill opacity=1 ][line width=0.75]      (0, 0) circle [x radius= 3.35, y radius= 3.35]   ;
\draw [shift={(424.58,249)}, rotate = 90.44] [fill={rgb, 255:red, 0; green, 0; blue, 0 }  ,fill opacity=1 ][line width=0.08]  [draw opacity=0] (12,-3) -- (0,0) -- (12,3) -- cycle    ;

% Text Node
\draw (94,76.4) node [anchor=north west][inner sep=0.75pt]  [font=\scriptsize]  {$( d+1,d-1)$};
% Text Node
\draw (185,100.4) node [anchor=north west][inner sep=0.75pt]  [font=\tiny]  {$d^{(2)}$};
% Text Node
\draw (162,102.4) node [anchor=north west][inner sep=0.75pt]  [font=\tiny]  {$\tilde{v}$};
% Text Node
\draw (227,116.4) node [anchor=north west][inner sep=0.75pt]  [font=\tiny]  {$\dbar$};
% Text Node
\draw (239,146.4) node [anchor=north west][inner sep=0.75pt]  [font=\tiny]  {$d^{(2)}$};
% Text Node
\draw (281,162.4) node [anchor=north west][inner sep=0.75pt]  [font=\tiny]  {$\dbar$};
% Text Node
\draw (284,201.4) node [anchor=north west][inner sep=0.75pt]  [font=\tiny]  {$d^{(2)}$};
% Text Node
\draw (326,217.4) node [anchor=north west][inner sep=0.75pt]  [font=\tiny]  {$\dbar$};
% Text Node
\draw (338,247.4) node [anchor=north west][inner sep=0.75pt]  [font=\tiny]  {$d^{(2)}$};
% Text Node
\draw (380,263.4) node [anchor=north west][inner sep=0.75pt]  [font=\tiny]  {$\dbar$};
% Text Node
\draw (481,278) node [anchor=north west][inner sep=0.75pt]   [align=left] {0};
% Text Node
\draw (420,308) node [anchor=north west][inner sep=0.75pt]   [align=left] {1};
% Text Node
\draw (367,308) node [anchor=north west][inner sep=0.75pt]   [align=left] {2};
% Text Node
\draw (432,263.4) node [anchor=north west][inner sep=0.75pt]  [font=\tiny]  {$\dbar$};
% Text Node
\draw (419,232) node [anchor=north west][inner sep=0.75pt]   [align=left] {0};
\end{tikzpicture}
                \caption{Zig-zag argument}
        \label{fig:Zig-Zag}
    \end{figure}
\end{proof}
\begin{table}[!h]
        \centering
\begin{tabular}{|p{0.12\textwidth}|p{0.12\textwidth}|p{0.12\textwidth}|p{0.12\textwidth}|p{0.12\textwidth}|}
\hline 
  & $\displaystyle \mathrm{ch}_{(1,1,1,1)}(3)$ & $\displaystyle \mathrm{ch}_{(3,1)}(3)$ & $\displaystyle\mathrm{ch}_{(2,2)}(3) $ & $\displaystyle \mathrm{ch}_{(4)}(4)$ \\
\hline 
 $\displaystyle \mathbb{v}_{(1,1,1,1)}( 3)$ & $\displaystyle \neq 0$ & $\displaystyle 0$ & $\displaystyle  0$ & $\displaystyle 0$ \\
\hline 
 $\displaystyle \mathbb{v}_{(3,1)}( 3)$ & $\displaystyle *$ & $\displaystyle \neq 0$ & $\displaystyle 0$ & $\displaystyle 0$ \\
\hline 
 $\displaystyle \mathbb{v}_{(2,2)}( 3)$ & $\displaystyle *$ & $\displaystyle *$ & $\displaystyle \neq 0$ & $\displaystyle 0$ \\
\hline 
 $\displaystyle \mathbb{v}_{(4)}( 3)$ & $\displaystyle *$ & $\displaystyle *$ & $\displaystyle *$ & $\displaystyle \neq 0$ \\
 \hline
\end{tabular}
\caption{Pairing matrix, $d=3$}
        \end{table}
        \begin{table}[!h]
        \centering
\begin{tabular}{|p{0.12\textwidth}|p{0.15\textwidth}|p{0.12\textwidth}|p{0.12\textwidth}|p{0.10\textwidth}|p{0.10\textwidth}|p{0.10\textwidth}|}
\hline 
  & $\displaystyle \mathrm{ch}_{(1,1,1,1,1)}(4)$ & $\displaystyle \mathrm{ch}_{(3,1,1)}(4)$ & $\displaystyle\mathrm{ch}_{(2,2,1)}(4) $ & $\displaystyle \mathrm{ch}_{(4,1)}(4)$ &$\displaystyle \mathrm{ch}_{(3,2)}(4)$& $\displaystyle \mathrm{ch}_{(5)}(4)$ \\
\hline 
 $\displaystyle \mathbb{v}_{(1,1,1,1,1)}( 4)$ & $\displaystyle \neq 0$ & $\displaystyle 0$ & $\displaystyle  0$ & $\displaystyle 0$& $\displaystyle 0$& $\displaystyle 0$ \\
\hline 
 $\displaystyle \mathbb{v}_{(3,1,1)}( 4)$ & $\displaystyle *$ & $\displaystyle \neq 0$ & $\displaystyle 0$ & $\displaystyle 0$ & $\displaystyle 0$& $\displaystyle 0$ \\
\hline 
 $\displaystyle \mathbb{v}_{(2,2,1)}( 4)$ & $\displaystyle *$ & $\displaystyle *$ & $\displaystyle \neq 0$ & $\displaystyle 0$& $\displaystyle 0$& $\displaystyle 0$  \\
\hline 
 $\displaystyle \mathbb{v}_{(4,1)}( 4)$ & $\displaystyle *$ & $\displaystyle *$ & $\displaystyle *$ & $\displaystyle \neq 0$ & $\displaystyle 0$& $\displaystyle 0$ \\
 \hline
 $\displaystyle \mathbb{v}_{(3,2)}( 4)$ & $\displaystyle *$ & $\displaystyle *$ & $\displaystyle *$ & $\displaystyle *$ & $\displaystyle \neq 0$& $\displaystyle 0$ \\
 \hline
  $\displaystyle \mathbb{v}_{(5)}( 4)$ & $\displaystyle *$ & $\displaystyle *$ & $\displaystyle *$ & $\displaystyle *$ & $\displaystyle *$& $\displaystyle \neq 0$ \\
 \hline
\end{tabular}
\caption{Pairing matrix, $d=4$}
\end{table}

For $I \subset \{1,\ldots,d\}$ we introduce the following notation
\begin{align}
  \mathrm{det}(z^{I},\mathsf{Eu}_{I})
& \define \sum\mathrm{sign}(\sigma)({z}^{\sigma(\ii_1)}\sum_{\ii\in I}z^{\ii}\partial_{\ii})\wedge\cdots\wedge ({z}
^{\sigma(\ii_l)}\sum_{\ii\in I}z^{\ii}\partial_{\ii}) \\
& =l!\cdot({z}^{\ii_1}\mathsf{Eu}_I)\wedge\cdots\wedge ({z}^{\ii_l}\mathsf{Eu}_I).
\end{align} 
Here, $\mathsf{Eu}_I$ is the Euler vector field in the directions $I$:
\begin{equation}\label{}
  \mathsf{Eu}_I \define \sum_{\ii \in I} z^{\ii} \del_{\ii} .
\end{equation}

Suppose that $\lambda = (\lambda_1,\ldots, \lambda_k)$, $\lambda_1 \geq \cdots \geq \lambda_k$ is a partition of the
integer $d+1$.
Such a $\lambda$ defines a set partition $\Lambda = (\Lambda_1,\ldots, \Lambda_k)$ of the set $\{0,\ldots,d\}$, with
$\lambda_i = |\Lambda_i|$ by
declaring: 
\begin{align*}
  \Lambda_1 & = \{0,1,\ldots, \lambda_1-1\} \\
  \Lambda_2 & = \{\lambda_1,\ldots, \lambda_1+\lambda_2-1\} \\
  & \vdots \\
  \Lambda_k & = \{\lambda_{k-1}, \ldots, d\} .
\end{align*}
When we write $\lambda$ for a partition of $d+1$ we always denote $\Lambda$ this particular set partition of
$\{0,\ldots,d\}$.
\begin{definition}\label{dfn:partitionv}
  For a partition $\lambda$ of $d+1$ with 
    \[
    \lambda_1 \geq \cdots \geq \lambda_k,\quad \lambda_1 \geq 2,
    \]
    define
    \[
      \mathbb{v}^{(d-1)}_{\lambda}(d) \define P(d) \partial_{\1}\wedge \mathbb{w}_\lambda(d) 
    \]
    where
    \begin{align*}
      \mathbb{w}_\lambda(d) &\define \sum\mathrm{sign}(\sigma)(z^{\1}{z}^{\sigma(\ii_1)}\mathsf{Eu}_{\Lambda_1-\{0\}}\wedge\cdots\wedge ({z}^{\sigma(\ii_l)}\mathsf{Eu}_{\Lambda_1-\{0\}})
                         \\ & \wedge \mathrm{det}(z^{\Lambda_2},\mathsf{Eu}_{\Lambda_2})\wedge\cdots \wedge\mathrm{det}(z^{\Lambda_{k-1}},\mathsf{Eu}_{\Lambda_{k-1}})\wedge\mathrm{det}(z^{\Lambda_{k}},\mathsf{Eu}_{\Lambda_{k}}).
    \end{align*}
    If $\lambda=(1,\dots,1)$, we define
  \begin{align*}
    \mathbb{v}^{(d-1)}_{(1,\dots,1)}(d) & \define P(d)\partial_{\1}\wedge \mathbb{w}_{(1,\dots,1)} (d) \\ & \; = P(d)
    \del_{\1} \wedge (z^{\1})^{3}\partial_{\1}\wedge (z^{\2})^{2}\partial_{\2}\wedge\cdots\wedge (z^{\dd})^{2}
    \partial_{\dd} .
  \end{align*}
\end{definition}

\begin{lemma}\label{lem:wlambda}
  For $d\geq2$ and $\lambda$ a partition of $\{1,\ldots,d\}$ one has:
    \begin{itemize}
      \item[(a)] The $\lie{w}_d$-chain $\mathbb{w}_{\lambda}(d)$ is $\bd^{(2)}$-closed.
      \item[(b)] There exists a $\lie{witt}_d$-chain $\mathbb{u}$ of total degree one such that $\bd^{(2)} \mathbb{v}
        ^{(d-1)}_{\lambda}(d) = \dbar
      \mathbb{u}$.
    \end{itemize}
\end{lemma}
\begin{proof}
  For part (a) we use the following elementary bracket identities
  \[
    [z^{\jj}\mathsf{Eu}_I,z^{\kkk}\mathsf{Eu}_I]=0 ,
  \]
  where $\jj,\kkk \in I$, and
  \[
    [z^\jj\mathsf{Eu}_I,z^\kkk\mathsf{Eu}_J]=0
  \]
  where $I\cap J=\varnothing$, $\jj\in I,\ \kkk\in J.$
  Every wedge factor appearing in $\mathbb{w}_\lambda(d)$ is of the form $z^{\jj}\mathsf{Eu}_I$ for one of the blocks $I=\Lambda_a$. 
  Thus any two wedge factors either belong to the same block, in which case the first identity applies, or to different blocks, in which case the second identity applies.

  Part (b). 
  Since $\mathbb{w}_\lambda (d)$ is $\bd^{(2)}$-closed, the only contributions to $\bd^{(2)}\mathbb{v}_\lambda^{(d-1)} (d)
  $ come from brackets of $P(d) \del_{\1}$ with the individual wedge factors of $\mathbb{w}_\lambda (d)$.
  Referring to the definition of $\op{det}(z^I, \mathsf{Eu}_I)$, it suffices to show the following type
  of bracket in $\lie{witt}_d$:
  \begin{equation}\label{}
    [P(d) \del_{\1} , f \mathsf{Eu}_I] 
  \end{equation}
  is $\dbar$ exact, where $f$ is an arbitrary homogeneous polynomial in $z$ which is at least linear.
  
  Using
  \[
    \mathsf{Eu}_I(P)=-d\,\tau_I P,
    \qquad
    \tau_I\define \sum_{\ii\in I}z^\ii x^{\ii},
  \]
  we compute
  \begin{equation}\label{eq:P-Eu-bracket-expanded}
    [P\del_\1,f\mathsf{Eu}_I]
    =
    P(\del_\1 f)\mathsf{Eu}_I
    +
    \mathbf 1_{\1\in I} fP\del_\1
    +
    d f\tau_I P\del_\1 .
  \end{equation}
  We argue that each term on the right-hand side is $\dbar$-exact. 
  The first and second terms are exact because their coefficients are holomorphic (depend only on $z$) positive-degree
multiples of $P$. The second term is exact for the same reason. 
For the third term, if $f=f(z)$ has degree $m>0$, then $f\tau_I P$ has Euler weight
$m-d$. 
Since $m \geq 1$, we see that $f\tau_I P$ is also $\dbar$-exact.
\end{proof}

By lemmas \ref{lem:descent} and \ref{lem:wlambda}, each class $\mathbb{v}^{(d-1)}_\lambda(d)$ lifts to a genuine cycle
\[
  \mathbb{v}_\lambda(d)
  \in C^{Lie}_\bullet(\lie{witt}_d)
\]
whose leading term is $\mathbb{v}^{(d-1)}_\lambda(d)$. 

\paragraph
Let
\[
  \op{Part}(\{0,\ldots,d\})^{-}
\]
denote the set of partitions of $\{0,\ldots,d\}$, ordered by block sizes, with
the partition of type $(2,1,\ldots,1)$ omitted. We order this set as follows.
First, we order partitions by increasing number of nonsingleton blocks, and
within each such stratum we use lexicographic order on the nonincreasing
sequence of block sizes. Equivalently, the order is chosen so that the examples
for $d=3$ and $d=4$ begin
\[
  (1,1,1,1)<(3,1)<(2,2)<(4)
\]
and
\[
  (1,1,1,1,1)<(3,1,1)<(2,2,1)<(4,1)<(3,2)<(5).
\]

\begin{proposition}\label{prop:upper-triangular-pairing}
  The pairing matrix
  \[
    \left(
    \big\langle \op{ch}_{\mu}(d),
    \til{\mathbb{v}}_{\lambda}(d)
    \big\rangle
    \right)_{\lambda,\mu\in \op{Part}(\{0,\ldots,d\})^{-}}
  \]
  is upper triangular with nonzero diagonal entries. More explicitly,
  \[
    \big\langle \op{ch}_{\lambda}(d),
    \til{\mathbb{v}}_{\lambda}(d)
    \big\rangle\neq 0
  \]
  and, if $\lambda<\nu$, then
  \[
    \big\langle \op{ch}_{\lambda}(d),
    \til{\mathbb{v}}_{\nu}(d)
    \big\rangle=0 .
  \]
\end{proposition}

\begin{proof}
  By lemmas \ref{lem:descent} and \ref{lem:wlambda}, we have shown that for each partition $\nu$ there is a cycle of the form
  \[
    \til{\mathbb{v}}_\nu(d)
    =
    \mathbb{v}^{(d-1)}_\nu(d)+\text{lower Jouanolou-degree terms}
  \]
  Importantly, the lower Jouanolou-degree terms do not contribute to the pairing with the Chern--Weil cocycles
  $\op{ch}_\lambda(d)$, for degree reasons. Hence it is enough to compute the
  pairing on the leading terms $\mathbb{v}^{(d-1)}_\nu(d)$.

  The leading term is
  \[
    \mathbb{v}^{(d-1)}_\nu(d)
    =
    P(d)\del_\1\wedge \mathbb{w}_\nu(d).
  \]
  Pairing the cocycle
  \[
    \op{ch}_\lambda
    =
    \op{ch}_1^{i_1}\cdots \op{ch}_d^{i_d}
  \]
  with this leading term amounts to distributing the trace factors among the wedge
  factors of $\mathbb{w}_\nu(d)$. 
  Each trace factor of length $r$ can be
  nonzero only if it is supported on a block of $\nu$ of size at least $r$.
  Moreover, the determinant expression
  \[
    \det(z^I,\mathsf{Eu}_I)
  \]
  forces the length-$|I|$ trace to use precisely the variables in the block
  $I$; if fewer or different variables are used, the alternating determinant
  gives zero.

  It follows that the pairing can be nonzero only when the block sizes required
  by $\lambda$ can be matched with the block sizes occurring in $\nu$. With the
  chosen order, this implies
  \[
    \lambda\leq \nu .
  \]
  Therefore
  \[
    \big\langle \op{ch}_{\lambda}(d),
    \mathbb{v}^{(d-1)}_{\nu}(d)
    \big\rangle=0
    \qquad \text{if } \lambda<\nu .
  \]

  On the diagonal, the matching is unique up to permutation of equal-sized
  blocks. In that case each determinant block is paired with the corresponding
  trace factor of the same length. The determinant antisymmetrization cancels
  the antisymmetrization in the trace and leaves a nonzero scalar multiple of
  \[
    \op{Res}\big(P(d)\,\d z^\1\cdots \d z^\dd\big) = 1
  \]
  Hence
  \[
    \big\langle \op{ch}_{\lambda}(d),
    \mathbb{v}^{(d-1)}_{\lambda}(d)
    \big\rangle\neq 0 .
  \]
  Since the correction terms do not affect the pairing, the same upper
  triangularity and nonvanishing diagonal hold for the cycles
  $\til{\mathbb{v}}_\lambda(d)$.
\end{proof}

\begin{theorem}\label{thm:cwd-injective}
  For every $d\geq 1$, the universal Chern--Weil homomorphism
  \[
    \mathsf{cw}_d\colon
    \C[\op{ch}_1,\ldots,\op{ch}_d]_{d+1}
    \to
    \HH^2_{\op{Lie}}(\lie{witt}_d)
  \]
  is injective.
\end{theorem}

\begin{proof}
  The case $d=1$ is classical, so assume $d\geq 2$.

  By Lemma \ref{lem:chd+1}, the class $\op{ch}_{d+1}$ lies in the span of the
  Chern--Weil classes
  \[
    \op{ch}_1^{i_1}\cdots \op{ch}_d^{i_d},
    \qquad
    \sum_{p=1}^d p\,i_p=d+1.
  \]
  Equivalently, we may replace the basis element corresponding to the partition
  $(2,1,\ldots,1)$ by the class $\op{ch}_{d+1}=\op{ch}_{(d+1)}$. Thus it is
  enough to prove linear independence of the classes
  \[
    \{\op{ch}_{\lambda}(d)\}_{\lambda\in
    \op{Part}(\{0,\ldots,d\})^{-}} .
  \]

  Suppose that there is a linear relation
  \[
    \sum_{\lambda\in \op{Part}(\{0,\ldots,d\})^{-}}
    a_\lambda\,\op{ch}_{\lambda}(d)=0
    \qquad
    \text{in }
    \HH^2_{\op{Lie}}(\lie{witt}_d).
  \]
  Pair this relation with the cycles
  $\til{\mathbb{v}}_\nu(d)$, ordered as in proposition
  \ref{prop:upper-triangular-pairing}. Let $\lambda_0$ be the smallest
  partition for which $a_{\lambda_0}\neq 0$. Pairing with
  $\til{\mathbb{v}}_{\lambda_0}(d)$ gives
  \[
    0
    =
    \sum_\lambda a_\lambda
    \big\langle
      \op{ch}_\lambda(d),
      \til{\mathbb{v}}_{\lambda_0}(d)
    \big\rangle .
  \]
  By upper triangularity, all terms with $\lambda>\lambda_0$ vanish, and by
  minimality all terms with $\lambda<\lambda_0$ have coefficient zero. Hence
  \[
    0
    =
    a_{\lambda_0}
    \big\langle
      \op{ch}_{\lambda_0}(d),
      \til{\mathbb{v}}_{\lambda_0}(d)
    \big\rangle .
  \]
  The diagonal pairing is nonzero, so $a_{\lambda_0}=0$, a contradiction.
  Therefore all $a_\lambda$ vanish, and the Chern--Weil classes are linearly
  independent.

  Hence $\mathsf{cw}_d$ is injective.
\end{proof}

\subsection{Injectivity for the $\lie{gl}_r$-extended Chern--Weil homomorphism}

Fix $0 \leq k \leq d$.  Let
\[
  \nu=(1^{i_1},2^{i_2},\ldots,(d-k)^{i_{d-k}})
\]
be a partition of $d-k$, and let $\kappa=(k_1,\ldots,k_r)$ satisfy
\[
  k_1+2k_2+\cdots+rk_r=k+1 .
\]
We write
\begin{equation}\label{eq:ch-nu-kappa}
  \op{ch}_{\nu;\kappa}(d,k)
  \define
  \op{ch}_1^{i_1}\cdots \op{ch}_{d-k}^{i_{d-k}}
  \lie{t}_1^{k_1}\cdots \lie{t}_r^{k_r}.
\end{equation}

Let
\[
  \theta_\kappa(A)
  \define
  \op{Tr}(A)^{k_1}\op{Tr}(A^2)^{k_2}\cdots\op{Tr}(A^r)^{k_r}.
\]
The polynomials $\theta_\kappa$, with $\sum_j j k_j=k+1$, form a basis of
$(S^{k+1}\lie{gl}_r^*)^{\lie{gl}_r}$.  Choose the dual basis
\[
  \mathbb{E}_\kappa \in (S^{k+1}\lie{gl}_r)^{\lie{gl}_r},
  \qquad
  \langle \theta_\kappa,\mathbb{E}_{\kappa'}\rangle=\delta_{\kappa,\kappa'}.
\]
Write
\[
  \mathbb{E}_\kappa
  =
  \sum_{\ell}
  A^{\kappa,\ell}_0\odot A^{\kappa,\ell}_1\odot\cdots\odot A^{\kappa,\ell}_k .
\]

\begin{definition}\label{dfn:partitionvv}
Let $\lambda=(\lambda_1,\ldots,\lambda_{d-k})$ be a partition of $d-k$, written with
$\lambda_1\geq\cdots\geq\lambda_{d-k}$.  Define the chain
\[
\begin{aligned}
\mathbb{v}_{\lambda,\kappa}(d)
\define
\sum_{\tau\in S_{k+1}}\sum_{\ell}
&
P(d)\,
A^{\kappa,\ell}_{\tau(0)}
\wedge z^{\1}A^{\kappa,\ell}_{\tau(1)}
\wedge\cdots\wedge
z^{\kkk}A^{\kappa,\ell}_{\tau(k)}
\\
&\wedge
\det(z^{\Lambda_1},\mathsf{Eu}_{\Lambda_1})
\wedge\cdots\wedge
\det(z^{\Lambda_{d-k}},\mathsf{Eu}_{\Lambda_{d-k}}).
\end{aligned}
\]
\end{definition}
\begin{lemma}\label{lem:wlambdaMixed}
For $d\geq 2$, the chain $\mathbb{v}_{\lambda,\kappa}(d)$ satisfies
\[
  \bd^{(2)}\mathbb{v}_{\lambda,\kappa}(d)
  =
  \dbar \mathbb{u}
\]
for some $\lie{witt}_d\ltimes \lie{gl}_r(\cJ_d)$-chain $\mathbb{u}$ of total degree one.
\end{lemma}

\begin{proof}
The $\lie{gl}_r(\cJ_d)$-part
\[
  \sum_{\tau\in S_{k+1}}\sum_{\ell}
  P(d)\,
  A^{\kappa,\ell}_{\tau(0)}
  \wedge z^{\1}A^{\kappa,\ell}_{\tau(1)}
  \wedge\cdots\wedge z^{\kkk}A^{\kappa,\ell}_{\tau(k)}
\]
is $\bd^{(2)}$-closed, because $\cJ_d$ is commutative.  The remaining factors are the same Witt chains as in
lemma~\ref{lem:wlambda}, so the same descent argument gives the desired $\mathbb{u}$.
\end{proof}

\begin{theorem}\label{thm:cwdr}
The $\lie{gl}_r$-extended universal Chern--Weil map
\[
  \mathsf{cw}_{d,r}
  \colon
  \C[y_1,\ldots,y_d,\lie{t}_1,\ldots,\lie{t}_r]_{d+1}
  \longrightarrow
  \HH^2_{\op{Lie}}
  \bigl(\lie{witt}_d\ltimes \lie{gl}_r(\cJ_d)\bigr)
\]
is injective.
\end{theorem}

\begin{proof}
Pair the classes $\op{ch}_{\nu;\kappa}(d,k)$ with the chains
$\mathbb{v}_{\lambda,\kappa'}(d)$.  By Lemma~\ref{lem:wlambdaMixed}, these chains can be lifted to cycles, so the pairings are homology invariants.

The $\lie{gl}_r$-factor separates the different $\kappa$'s:
\[
  \big\langle \theta_\kappa,\mathbb{E}_{\kappa'}\big\rangle
  =
  \delta_{\kappa,\kappa'}.
\]
Thus the pairing matrix is block diagonal with respect to $\kappa$.  Inside each block, the Witt part is exactly the
pairing matrix appearing in the proof of lemma~\ref{lem:wlambda}; with the same ordering on partitions, it is upper triangular with nonzero diagonal entries.  Therefore the full pairing matrix is upper triangular with nonzero diagonal.  Hence no nontrivial linear combination of the classes
$\op{ch}_{\nu;\kappa}(d,k)$ lies in the kernel of $\mathsf{cw}_{d,r}$, proving injectivity.
\end{proof}

\section{Cohomology of the Witt algebra}\label{s:gf}

In this section we finish the classification of central extensions of the dg Witt algebra $\lie{witt}_d$ by proving that
the Chern--Weil map \eqref{eq:cw1} is an isomorphism.
We have shown in the last section that that this map is injective.
In this section, we prove the map is an isomorphism by computing the second Lie algebra (hyper) cohomology of $\lie{witt}
_d$ using a pair of filtrations which are inspired from \cite[section 2.4.1]{FuksBook} as well as \cite{HKgf}.

It is a useful time to spell out the topology we use on the dg Witt algebra. 
The underlying complex of $\lie{witt}_d$ lies in the dg category associated to the quasi-abelian category $\mathrm{ILC}_\C$, i.e. the category of inductive limits of linearly compact complex vector spaces (as discussed in \cite[section 4.1]{FHK}).
For us, the cochains which comprise $C^\bu_{Lie}(\lie{witt}_d)$ are continuous; meaning explicitly that we consider functionals which are continuous on each linearly compact subspace of $\lie{witt}_d$.
The argument we present below for computing the (continous) Lie algebra cohomology of $\lie{witt}_d$ applies, without any significant modifications, to compute the \textit{algebraic} Lie algebra cohomology of $\lie{witt}_d^{poly}$ (meaning, no extra condition on cochains).
In fact, since $GL_d$-invariant cochains are automatically continuous, there is no essential difference, homotopically, between the algebraic and continuous settings.

\subsection{The diagonal and order filtrations}

The main computational tool is a filtration, which we call \textit{diagonal}, on the Chevalley-Eilenberg complex
computing the continous Lie algebra (hyper)cohomology of
$\lie{witt}_d$.
The idea for this filtration originates in the proof \cite[theorem 2.4.1a,b]{FuksBook} as well as more recently in
\cite{HKgf}.
As we define it, this diagonal filtration is on the $GL_d$-invariant continuous cochain complex
\begin{equation}\label{}
  C_{(0)}^{\bu} \define C^\bu_{Lie}(\lie{witt}_d)^{GL_d}
\end{equation}
which is quasi-isomorphic to the full Chevalley-Eilenberg complex by reductivity.

Fix an integer $n$ and denote by $C^n_{(0)}$ the graded vector space of $n$-linear cochains inside of $C^\bu_{(0)}$.
We introduce first the "diagonal" filtration on $C^n_{(0)}$ which measures how many distinct points a cochain is
sensitive to.
Passing to the $k$th associated graded will allow us to isolate cochains which see precisely $k$ distinct points, and
the residual jet filtration on this associated graded will measure the order of vanishing in the normal direction to
the corresponding diagonal stratum.
With this sketch of the idea, we proceed to introduce some notations and after we will provide the precise definitions of the filtrations.

\paragraph
Let $\pi$ be a partition of the set $I=\{1,\dots,n\}$.
We can view it as a equivalence relation on $I$ by declaring that $i\sim_\pi j$ if $i,j$ lie in the same block.
Define the following ideal
\[
  \fI_\pi \define \big(z^{\s}_i-z^{\s}_j|i\sim_\pi j \; , \quad \s=\1,\dots,\dd \big) \subset \C\llbracket \mathbf{z}
  \rrbracket^{\otimes n},
\]
where 
\[
  \C\llbracket \mathbf{z} \rrbracket = \C\llbracket z_1,\ldots,z_n \rrbracket = \C \llbracket z_1^\1, \ldots ,
  z_1^{\dd} ; \ldots ; z_n^{\1} , \ldots, z_n^{\dd} \rrbracket
\]

For a partition $\pi$, define
\[
  \lie{witt}_d^{\otimes_\pi n}
  \define
  \lie{witt}_d^{\otimes n}/\fI_\pi \lie{witt}_d^{\otimes n}.
\]
Equivalently, if $B$ runs over the blocks of $\pi$, then
\[
  \lie{witt}_d^{\otimes_\pi n}
  \cong
  \bigotimes_{B\in \pi}
  \lie{witt}_d^{\otimes_{\C\llbracket z\rrbracket}|B|},
\]
where inside each block all formal coordinates have been identified.
We also write
\[
  \widehat{\lie{witt}_d^{\otimes_\pi n}}
  \define
  \lim_{\longleftarrow N}
  \frac{\lie{witt}_d^{\otimes n}}
  {\fI_\pi^{N+1}\lie{witt}_d^{\otimes n}}
\]
for the completion along the partial diagonal $\Delta_\pi$.

Next define the lower diagonal inside $\Delta_\pi$. Let
\[
  \fI_{\pi,k-1}
  \define
  \bigcap_{\substack{\rho\in \operatorname{Part}(I,k-1)\\ \rho\leq \pi}}
  \frac{\fI_\rho}{\fI_\pi}
  \subset
  \C\llbracket z_1,\ldots,z_n\rrbracket/\fI_\pi ,
\]
where $\rho\leq \pi$ means that $\rho$ is obtained from $\pi$ by merging two blocks. This ideal cuts out the union of lower diagonals inside $\Delta_\pi$, namely the locus where two distinct $\pi$-blocks collide.
We write
\[
  \left(
  \lie{witt}_d^{\otimes_\pi n}
  \right)^{\widehat\Delta_{\pi,k-1}}
  \define
  \lim_{\longleftarrow M}
  \frac{\lie{witt}_d^{\otimes_\pi n}}
  {\fI_{\pi,k-1}^{M+1}\lie{witt}_d^{\otimes_\pi n}}
\]
for the completion of $\lie{witt}_d^{\otimes_\pi n}$ along this lower diagonal.

\paragraph
Let $\mathrm{Part}(I,k)$ be the set of partitions of $I$ with exactly $k$-blocks.
Define
\[
  \Delta_{k}^{\leq N} C_{(0)}^{n,\otimes} \define \mathrm{Hom}_{\lie{gl}_d}\left(\frac{\lie{witt}^{\otimes n}_d}
{\cap_{\pi\in \mathrm{Part}(I,k)}\fI^{N+1}_{\pi}\lie{witt}^{\otimes n}_d} \; , \;\C\right)
\]
Let
\[
\Delta_{k}^{\leq N}C_{(0)}^{n} \define \op{Alt}_n \big(\Delta_{k}^{\leq N}C_{(0)}^{n,\otimes} \big)
\]
be the totally (graded) alternating subspace.
It is a subspace of space of $n$-linear continuous cochains on $\lie{witt}_d$.
We set $\Delta_{0}C_{(0)}^{n}=0$ and define
\[
\Delta_{k}C_{(0)}^{n} \define \bigcup_{N\geq 0}\Delta_{k}^{\leq N}C_{(0)}^{n}.
\]
for $k > 0$.
The $k$ here measures the ``multi-locality" of the invariant cochain.
This defines an increasing filtration on the space of $n$-linear continuous cochains:
\begin{equation}\label{}
  C^n_{(0)} = \Delta_n C^n_{(0)} \supset \Delta_{n-1} C^n_{(0)} \supset \cdots \supset \Delta_{k} C_{(0)}^n \supset\Delta_{k-1} C_{(0)}^n \supset \cdots
 \supset \Delta_1 C_{(0)}^n \supset 0 .
\end{equation}
The first piece $\Delta_1 C^n_{(0)}$ is the space of \textit{local} $n$-linear cochains.
That is, those cochains supported at a single point.
This filtration is exhaustive because we look at the space of $GL_d$-invariant cochains.

\begin{lemma}\label{lem:diag}
For each $n$ and $1\leq k\leq n$, there is a natural identification
\[
  \operatorname{gr}^\Delta_k C^n_{(0)}
  \cong
  \operatorname{Alt}_n
  \left(
  \bigoplus_{\pi\in \operatorname{Part}(I,k)}
  \frac{
  \operatorname{Hom}_{\lie{gl}_d}
  \left(
    \left(\Hat{\lie{witt}_d^{\otimes_\pi n}}\right),
  \C
  \right)}
  {
  \operatorname{Hom}_{\lie{gl}_d}
  \left(
  \left(\lie{witt}_d^{\otimes_\pi n}\right)^{\widehat\Delta_{\pi,k-1}},
  \C
  \right)}
  \right).
\]
\end{lemma}
\begin{proof}
We first ignore the alternating operation. 
An element of $\Delta_k^{\leq N} C_{(0)}^{n,\otimes}$ is a $\lie{gl}_d$-invariant $n$-linear cochain whose support is contained, to order $N$, in the union of the partial diagonals $\Delta_\pi$ with $\pi\in \operatorname{Part}(I,k)$.

The quotient
\[
  \frac{\lie{witt}_d^{\otimes n}}
  {\bigcap_{\pi\in \operatorname{Part}(I,k)}
  \fI_\pi^{N+1}\lie{witt}_d^{\otimes n}}
\]
is the algebraic object corresponding to the $N$th infinitesimal neighborhood of the union
\[
  \bigcup_{\pi\in \operatorname{Part}(I,k)}\Delta_\pi .
\]
The different partial diagonals $\Delta_\pi$ with $k$ blocks meet only along lower diagonals, namely along loci where some distinct blocks collide. These intersections are contained in the union of partial diagonals with at most $k-1$ blocks. Therefore, after quotienting by $\Delta_{k-1}C_{(0)}^{n,\otimes}$, the contribution of the union of $k$-block diagonals splits as a direct sum of the contributions from the individual strata $\Delta_\pi$, with their lower-diagonal overlaps removed.

For fixed $\pi\in \operatorname{Part}(I,k)$, the contribution of $\Delta_\pi$ is the continuous dual of the formal completion
\[
  \widehat{\lie{witt}_d^{\otimes_\pi n}}
  =
  \lim_{\longleftarrow N}
  \frac{\lie{witt}_d^{\otimes n}}
  {\fI_\pi^{N+1}\lie{witt}_d^{\otimes n}}.
\]
The part already supported on lower diagonals inside $\Delta_\pi$ is the continuous dual of
\[
  \left(
  \lie{witt}_d^{\otimes_\pi n}
  \right)^{\widehat\Delta_{\pi,k-1}}
  =
  \lim_{\longleftarrow M}
  \frac{\lie{witt}_d^{\otimes_\pi n}}
  {\fI_{\pi,k-1}^{M+1}\lie{witt}_d^{\otimes_\pi n}}.
\]
Hence the contribution of the stratum indexed by $\pi$, modulo its lower-diagonal part, is
\[
  \frac{
  \operatorname{Hom}_{\lie{gl}_d}
  \left(
  \widehat{\lie{witt}_d^{\otimes_\pi n}},
  \C
  \right)}
  {
  \operatorname{Hom}_{\lie{gl}_d}
  \left(
  \left(
  \lie{witt}_d^{\otimes_\pi n}
  \right)^{\widehat\Delta_{\pi,k-1}},
  \C
  \right)}.
\]
Summing over $\pi\in \operatorname{Part}(I,k)$ gives the associated graded of the nonalternating diagonal filtration:
\[
  \operatorname{gr}^\Delta_k C_{(0)}^{n,\otimes}
  \cong
  \bigoplus_{\pi\in \operatorname{Part}(I,k)}
  \frac{
  \operatorname{Hom}_{\lie{gl}_d}
  \left(
  \widehat{\lie{witt}_d^{\otimes_\pi n}},
  \C
  \right)}
  {
  \operatorname{Hom}_{\lie{gl}_d}
  \left(
  \left(
  \lie{witt}_d^{\otimes_\pi n}
  \right)^{\widehat\Delta_{\pi,k-1}},
  \C
  \right)}.
\]

Finally, the diagonal filtration is $S_n$-equivariant, and the Chevalley--Eilenberg cochains are obtained from the
nonalternating cochains by applying total graded antisymmetrization. Since $\operatorname{Alt}_n$ commutes with passing
to the associated graded of this $S_n$-equivariant filtration, we obtain
\[
  \operatorname{gr}^\Delta_k C^n_{(0)}
  \cong
  \operatorname{Alt}_n
  \left(
  \operatorname{gr}^\Delta_k C_{(0)}^{n,\otimes}
  \right),
\]
which gives the claimed formula.
\end{proof}

\paragraph
Let $\pi\in \operatorname{Part}(I,k)$, and let $B_1,\ldots,B_k$ be its blocks. For each block $B_j$, define
\[
  \lie q_{\Delta_j}
  \subset
  \lie q^{\oplus |B_j|}
\]
to be the diagonal copy of $\lie q$. 
For $m\geq 0$, define
\[
  \lie N^m_\pi
  \define
  \bigoplus_{m_1+\cdots+m_k=m}
  \bigotimes_{j=1}^k
  \operatorname{Sym}^{m_j}
  \left(
  \lie q^{\oplus |B_j|}/\lie q_{\Delta_j}
  \right).
\]

\begin{lemma}\label{lem:normaljet}
Let $\pi\in \operatorname{Part}(I,k)$. Then the $m$th associated graded piece for the $\fI_\pi$-adic filtration is naturally
\[
  \frac{
  \fI_\pi^m\lie{witt}_d^{\otimes n}
  }{
  \fI_\pi^{m+1}\lie{witt}_d^{\otimes n}
  }
  \cong
  \lie N^m_\pi
  \otimes
  \lie{witt}_d^{\otimes_\pi n}.
\]
Equivalently, the degree $m$ normal jets to the partial diagonal $\Delta_\pi$ are given by $\lie N^m_\pi$.
\end{lemma}

\begin{proof}
The associated graded of the $\fI_\pi$-adic filtration is
\[
  \operatorname{gr}_{\fI_\pi}
  \left(
  \lie{witt}_d^{\otimes n}
  \right)
  =
  \bigoplus_{m\geq 0}
  \frac{
  \fI_\pi^m\lie{witt}_d^{\otimes n}
  }{
  \fI_\pi^{m+1}\lie{witt}_d^{\otimes n}
  }.
\]
The usual identification of the associated graded of the ideal filtration with the symmetric algebra on the cononormal
gives the a $\lie{gl}_d$-equivariant identification
\[
  \operatorname{gr}_{\fI_\pi}
  \left(
  \C\llbracket z_1,\ldots,z_n\rrbracket
  \right)
  \cong
  \operatorname{Sym}_{\C\llbracket \mathbf{z}\rrbracket^{\otimes_\pi n}}
  \left(
  \fI_\pi/\fI_\pi^2
  \right).
\]
Here
\[
  \C\llbracket \mathbf{z} \rrbracket^{\otimes_\pi n}
  =
  \C\llbracket z_1,\ldots,z_n\rrbracket/\fI_\pi
\]
is the coordinate ring of $\Delta_\pi$.

For a block $B_j$, the ideal identifying the coordinates inside $B_j$ has conormal module 
\[
  \left(
  \lie q^{\oplus |B_j|}/\lie q_{\Delta_j}
  \right)^* .
\]
In total, the normal directions to the ideal corresponding to $\pi$ split as the direct sum over the blocks:
\[
  \bigoplus_{j=1}^k
  \left(
  \lie q^{\oplus |B_j|}/\lie q_{\Delta_j}
  \right).
\]
The $m$th symmetric power of this normal space is precisely 
\[
  \lie N^m_\pi = \operatorname{Sym}^m
  \left(
  \bigoplus_{j=1}^k
  \lie q^{\oplus |B_j|}/\lie q_{\Delta_j}
  \right)
  \cong
  \bigoplus_{m_1+\cdots+m_k=m}
  \bigotimes_{j=1}^k
  \operatorname{Sym}^{m_j}
  \left(
  \lie q^{\oplus |B_j|}/\lie q_{\Delta_j}
  \right).
\]

Finally, restricting $\lie{witt}_d^{\otimes n}$ to the partial diagonal $\Delta_\pi$ gives
\[
  \lie{witt}_d^{\otimes n}/\fI_\pi\lie{witt}_d^{\otimes n}
  =
  \lie{witt}_d^{\otimes_\pi n}.
\]
Since $\lie{witt}_d^{\otimes n}$ is a free module over
$\C\llbracket z_1,\ldots,z_n\rrbracket$, tensoring the associated graded of the coordinate ring with the restriction of $\lie{witt}_d^{\otimes n}$ gives
\[
  \frac{
  \fI_\pi^m\lie{witt}_d^{\otimes n}
  }{
  \fI_\pi^{m+1}\lie{witt}_d^{\otimes n}
  }
  \cong
  \lie N^m_\pi
  \otimes
  \lie{witt}_d^{\otimes_\pi n}.
\]
This proves the claim.
\end{proof}

\paragraph 
The $k$th associated graded of the diagonal filtration
\[
  \mathrm{gr}^{\Delta}_kC_{(0)}^{n}
  = \frac{\Delta_{k}C_{(0)}^{n}}{\Delta_{k-1}C_{(0)}^{n}}
\]
carries an additional filtration by jet order.
Precisely, we index the jet-order by the complementary degree $q$ to get a decreasing filtration; for $q = 0, 1, \ldots$ define
\[
  F^q\big(
  \mathrm{gr}^{\Delta}_kC_{(0)}^{n}\big)
    \define \frac{\Delta_{k}^{\leq n-q} C_{(0)}^{n}}{\Delta_{k-1}
C_{(0)}^{n} \cap \Delta_{k}^{\leq n-q} C^n_{(0)}} =
\frac{\Delta_{k}^{\leq n-q} C_{(0)}^{n}+\Delta_{k-1} C_{(0)}^{n}}{\Delta_{k-1} C^n_{(0)}}
\]
Denote by $\op{gr}^{\op{jet}}_q \op{gr}_k^{\Delta} C^n_{(0)}$ the $q$th associated graded of this secondary filtration.

\begin{lemma}\label{lem:grgr}
Let $1\leq k\leq n$ and $q\geq 0$. 
There is a natural identification
\[
  \operatorname{gr}^{\operatorname{jet}}_q
  \operatorname{gr}^{\Delta}_k C^n_{(0)}
  \cong
  \operatorname{Alt}_n
  \left(
  \bigoplus_{\pi\in \operatorname{Part}(I,k)}
  \frac{
  \operatorname{Hom}_{\lie{gl}_d}
  \left(
  \lie N^{n-q}_\pi
  \otimes
  \lie{witt}_d^{\otimes_\pi n},
  \C
  \right)}
  {
  \operatorname{Hom}_{\lie{gl}_d}
  \left(
  \lie N^{n-q}_\pi
  \otimes
  \left(
  \lie{witt}_d^{\otimes_\pi n}
  \right)^{\widehat\Delta_{\pi,k-1}},
  \C
  \right)}
  \right).
\]
\end{lemma}
\begin{proof}
The only point to check is how the jet filtration on cochains is related to the $\fI_\pi$-adic filtration appearing in
lemma \ref{lem:normaljet}.

Fix a partition $\pi\in \operatorname{Part}(I,k)$ and set
\[
  M=\lie{witt}_d^{\otimes n},
  \qquad
  I=\fI_\pi .
\]
The $N$th infinitesimal neighborhood of the partial diagonal $\Delta_\pi$ is represented by
\[
  M/I^{N+1}M .
\]
Therefore the order-$N$ associated graded piece of the jet filtration on $n$-cochains is of the form 
\[
  \frac{
  \operatorname{Hom}_{\lie{gl}_d}
  \left(
  M/I^{N+1}M,\C
  \right)}
  {
  \operatorname{Hom}_{\lie{gl}_d}
  \left(
  M/I^NM,\C
  \right)}.
\]
The short exact sequence dual to
\[
  0
  \to
  I^NM/I^{N+1}M
  \to
  M/I^{N+1}M
  \to
  M/I^NM
  \to
  0
\]
remains exact after taking $\lie{gl}_d$-invariants 
Thus
\[
  \frac{
  \operatorname{Hom}_{\lie{gl}_d}
  \left(
  M/I^{N+1}M,\C
  \right)}
  {
  \operatorname{Hom}_{\lie{gl}_d}
  \left(
  M/I^NM,\C
  \right)}
  \cong
  \operatorname{Hom}_{\lie{gl}_d}
  \left(
  I^NM/I^{N+1}M,\C
  \right).
\]
By lemma~\ref{lem:normaljet}, we have
\[
  I^NM/I^{N+1}M
  \cong
  \lie N_\pi^N\otimes \lie{witt}_d^{\otimes_\pi n}.
\]
Consequently the order-$N$ jet associated graded along $\Delta_\pi$ is
\[
  \operatorname{Hom}_{\lie{gl}_d}
  \left(
  \lie N_\pi^N\otimes \lie{witt}_d^{\otimes_\pi n},
  \C
  \right).
\]

In the notation of the jet filtration on
$\operatorname{gr}^\Delta_k C^n_{(0)}$, we set $N=n-q$. Thus the $q$th jet associated graded corresponds to normal order $n-q$ along the partial diagonal.
We apply lemma \ref{lem:diag} to arrive at the desired expression.
\end{proof}

\paragraph
As a $\lie{gl}_d$-representation, we write
\[
  \lie{witt}_d
  =
  \cJ_d\otimes \lie q,
\]
where $\cJ_d$ denotes the formal function factor and $\lie q$ is the defining representation. Thus
\[
  \lie{witt}_d^{\otimes_\pi n}
  =
  \cJ_d^{\otimes_\pi n}\otimes \lie q^{\otimes n}.
\]
It follows from lemma~\ref{lem:grgr} that
\[
\operatorname{gr}^{\operatorname{jet}}_q
\operatorname{gr}^{\Delta}_k C^n_{(0)}
\cong
\operatorname{Alt}_n
\left(
\bigoplus_{\pi\in \operatorname{Part}(I,k)}
\operatorname{Hom}_{\lie{gl}_d}
\left(
\lie N_\pi^{n-q}\otimes \lie q^{\otimes n},
\,
\frac{
\operatorname{Hom}\big(\cJ_d^{\otimes_\pi n},\C\big)}
{
\operatorname{Hom}\big((\cJ_d^{\otimes_\pi n})^{\widehat\Delta_{\pi,k-1}},\C\big)}
\right)
\right).
\]
Here all linear duals are continuous duals. In particular, the duals of the completed $\cJ_d$-modules decompose as direct sums of finite-dimensional $\lie{gl}_d$-representations.

Recall that $\lie N_\pi^{n-q}$ is the degree $n-q$ part of the algebra of polynomial functions in the normal directions to the product of small diagonals determined by $\pi$. More explicitly, if $B_1,\ldots,B_k$ are the blocks of $\pi$, then
\[
  \lie N_\pi^m
  =
  \bigoplus_{m_1+\cdots+m_k=m}
  \bigotimes_{j=1}^k
  \operatorname{Sym}^{m_j}
  \left(
  \left(
  \lie q^{\oplus |B_j|}/\lie q_{\Delta_j}
  \right)^\vee
  \right).
\]
We compute this representation using the Koszul resolution of the quotient by the blockwise diagonal subspaces. Namely, for each $m$ there is a resolution
\[
  \lie N_\pi^m
  \simeq
  \left(\fK^{m,\bullet}_\pi,\d_K\right),
\]
where
\[
  \fK^{m,-r}_\pi
  =
  \bigoplus_{\substack{
  m_1+\cdots+m_k=m\\
  r_1+\cdots+r_k=r
  }}
  \bigotimes_{j=1}^k
  \left(
  \Lambda^{r_j}\lie q^\vee_{\Delta_j}
  \otimes
  \operatorname{Sym}^{m_j-r_j}
  \left(
  (\lie q^{\oplus |B_j|})^\vee
  \right)
  \right).
\]
Here $\lie q^\vee_{\Delta_j}$ denotes the linear forms cutting out the diagonal copy of $\lie q$ inside $\lie q^{\oplus |B_j|}$. Equivalently, it is the kernel of the restriction map
\[
  (\lie q^{\oplus |B_j|})^\vee
  \to
  \lie q^\vee.
\]
The differential is the usual Koszul differential. If
\[
  \alpha_j=x_1\wedge\cdots\wedge x_{r_j}
  \in \Lambda^{r_j}\lie q^\vee_{\Delta_j},
\]
then
\[
\begin{aligned}
&\d_K\left(
(\alpha_1\otimes f_1)\otimes\cdots\otimes(\alpha_k\otimes f_k)
\right) \\
&\define
\sum_{j=1}^k
(-1)^{r_1+\cdots+r_{j-1}}
(\alpha_1\otimes f_1)\otimes\cdots\otimes
\left[
\sum_{a=1}^{r_j}
(-1)^{a+1}
\left(
\iota_{x_a}\alpha_j
\otimes
x_a f_j
\right)
\right]
\otimes\cdots\otimes
(\alpha_k\otimes f_k).
\end{aligned}
\]
This is the Koszul differential for the regular sequence of linear functions cutting out the blockwise diagonal.

Substituting this resolution into the previous expression gives a complex quasi-isomorphic to
\[
\operatorname{Alt}_n
\left(
\bigoplus_{\pi\in \operatorname{Part}(I,k)}
\operatorname{Hom}_{\lie{gl}_d}
\left(
\fK^{n-q,\bullet}_\pi\otimes \lie q^{\otimes n},
\,
\frac{
\operatorname{Hom}\big(\cJ_d^{\otimes_\pi n},\C\big)}
{
\operatorname{Hom}\big((\cJ_d^{\otimes_\pi n})^{\widehat\Delta_{\pi,k-1}},\C\big)}
\right)
\right).
\]
The Chevalley--Eilenberg differential only contains brackets among inputs lying in the same block of $\pi$; brackets between different blocks vanish in the associated graded for the diagonal filtration. Therefore the $\lie q$- and normal-direction part identifies, block by block, with the reduced continuous Chevalley--Eilenberg complex of $\lie w_d$. Thus the preceding expression may be rewritten as
\[
\operatorname{Alt}_n
\left(
\bigoplus_{\pi\in \operatorname{Part}(I,k)}
\frac{
\operatorname{Hom}
\left(
\cJ_d^{\otimes_\pi n}\otimes
\bigotimes_{B\in\pi}\Lambda^\bullet \lie q^\vee,
\C
\right)}
{
\operatorname{Hom}
\left(
(\cJ_d^{\otimes_\pi n})^{\widehat\Delta_{\pi,k-1}}\otimes
\bigotimes_{B\in\pi}\Lambda^\bullet \lie q^\vee,
\C
\right)}
\otimes
\overline C^\bullet_{\mathrm{Lie}}(\lie w_d)^{\otimes \pi}
\right)^{\lie{gl}_d}.
\]
Here $\overline C^\bullet_{\mathrm{Lie}}(\lie w_d)^{\otimes \pi}$ means the tensor product over the blocks of $\pi$.

Next, multiplication in each block gives a natural quasi-isomorphism
\[
  \cJ_d^{\otimes_\pi n}
  =
  \bigotimes_{B\in\pi}
  \cJ_d^{\otimes_{\C\llbracket z\rrbracket}|B|}
  \to
  \bigotimes_{B\in\pi}\cJ_d.
\]
Indeed, from proposition \ref{prop:flatness} the dg ring $\cJ_d$ is flat over $\C\llbracket z\rrbracket$, and the multiplication map
\[
  \cJ_d\otimes_{\C\llbracket z\rrbracket}\cJ_d
  \to
  \cJ_d
\]
is a quasi-isomorphism.
Hence, after choosing an ordering of the blocks of $\pi$, the above complex is quasi-isomorphic to
\[
\operatorname{Alt}_k
\left(
\frac{
\operatorname{Hom}
\left(
\cJ_d^{\otimes k}\otimes (\Lambda^\bullet\lie q^\vee)^{\otimes k},
\C
\right)}
{
\operatorname{Hom}
\left(
(\cJ_d^{\otimes k})^{\widehat\Delta_{k-1}}\otimes
(\Lambda^\bullet\lie q^\vee)^{\otimes k},
\C
\right)}
\otimes
\overline C^\bullet_{\mathrm{Lie}}(\lie w_d)^{\otimes k}
\right)^{\lie{gl}_d}.
\]

The quotient
\[
  \frac{
\operatorname{Hom}
\left(
\cJ_d^{\otimes k}\otimes (\Lambda^\bullet\lie q^\vee)^{\otimes k},
\C
\right)}
{
\operatorname{Hom}
\left(
(\cJ_d^{\otimes k})^{\widehat\Delta_{k-1}}\otimes
(\Lambda^\bullet\lie q^\vee)^{\otimes k},
\C
\right)}
\]
is the algebraic model for relative de Rham current homology
\[
  H^{rel}_\bullet(k,k-1)
  \simeq
  H^{dR}_\bullet
  \left(
  \mathring D_d^k,
  \mathring D_d^{k-1}
  \right).
\]
More precisely, it is the quotient complex associated to the inclusion of currents supported on the lower diagonal into currents supported on the $k$-fold configuration. Equivalently, it is quasi-isomorphic to the mapping cone of this inclusion, and hence computes the relative current homology of the pair.
Notice that we use cohomological conventions so the $k$th homology sits in cohomological degree $-k$.

We will use the following three facts about this relative homology:
\begin{enumerate}
  \item For $k=1$, $H^{\mathrm{rel}}_\bullet(1,0)$ is one-dimensional and concentrated in cohomological degree $-2d+1$.
  \item For $k>1$, $H^{\mathrm{rel}}_\bullet(k,k-1)$ is concentrated in cohomological degrees greater than or equal to
    $k(-2d+1)$.
  \item The $\lie w_d$-action on $H^{\mathrm{rel}}_\bullet(k,k-1)$ is trivial. This follows from the Cartan homotopy formula.
\end{enumerate}
It follows that the associated graded contribution is
\begin{equation}\label{eq:E1gr}
  \operatorname{Alt}_k
  \left(
    H^{rel}_\bullet(k,k-1)
  \otimes
  \overline H^\bullet_{Lie}(\lie w_d)^{\otimes k}
  \right).
\end{equation}

We now inspect total degree two cohomology $\HH^2_{Lie}(\lie{witt}_d)$. 
The diagonal filtration gives a spectral sequence whose differentials have
cohomological degree $+1$ and increase the diagonal index.
We use the further jet filtration to compute the $E_1$ page, and have expressed its jet-associated graded in
\eqref{eq:E1gr}. 
For $k=1$, the relative current homology is nontrivial in cohomological degree $-2d+1$,
so the Lie algebra cohomology factor must contribute degree $2d+1$. 
The $k=1$ contribution is thus
\[
  H^{2d+1}_{Lie}(\lie w_d).
\]

For $k>1$, the total cohomological degree is bounded below by
\[
  k(2d+1)+k(-2d+1)=2k>2.
\]
Thus no term with $k>1$ contributes to total degree two.
Hence the $k=1$ term survives to $E_\infty$ and the spectral sequence converges.
We arrive at the following.

\begin{theorem}\label{thm:local}
The total degree two part of the diagonal spectral sequence converging to the Lie algebra hypercohomology of the
$\lie{witt}_d$ is concentrated in the local $k=1$ term. Consequently,
\[
  \HH^2_{Lie}(\lie{witt}_d)
  \cong
  H^2(C^\bullet_{(0)})
  \cong
  H^{2d+1}_{Lie}(\lie w_d).
\]
\end{theorem}

It is a classical result of Fuks that the Lie algebra cohomology of $\lie{w}_d$ is isomorphic to the de Rham cohomology
of the total space $Y_d$ of the restriction of the universal $U(d)$-bundle to the $2d$-skeleton of $BGL_d = Gr_{d,
\infty}$:
\begin{equation}\label{}
  \begin{tikzcd} 
    Y_d \ar[r] \ar[d] & \star \ar[d] \\
    \op{sk}_{2d} \ar[r] & BGL_d .
  \end{tikzcd}
\end{equation}
By the degeneration of the Serre spectral sequence we therefore have 
\begin{equation}\label{}
  H^{2d+1}_{Lie}(\lie{w}_d) \cong H^{2d+1} (Y_d) \cong H^{2d+2}(BGL_d) .
\end{equation}
We have already proved in theorem \ref{thm:cwd-injective} that the universal Chern--Weil homomorphism 
\begin{equation}\label{}
  \mathsf{cw}_{d} \colon \C[y_1,\ldots,y_d]_{d+1} \hookrightarrow \HH^{2}_{Lie}(\lie{witt}_d)
\end{equation}
is injective. 
We now know since dimensions match that this is an isomorphism.

\paragraph
The preceding proof extends directly to the case $\lie{witt}_d\ltimes \lie{gl}_r(\cJ_d)$. In fact, we can use the Serre–Hochschild spectral sequence. Define the filtration
\[
F^p_{\operatorname{Serre-Hoch}}C^n_{(0)}:=\{c\in C^n_{(0)}\mid c(\xi_1,\dots,\xi_n)=0\ \text{for}\ \xi_1,\dots,\xi_{n-p+1}\in \lie{witt}_d\}.
\]
This yields
\[
\operatorname{gr}^{\operatorname{Serre-Hoch}}_p\operatorname{gr}^{\operatorname{jet}}_q
\operatorname{gr}^{\Delta}_k C^n_{(0)}
\cong
\operatorname{Alt}_n
\left(
\bigoplus_{\pi\in \operatorname{Part}(I,k)}
\operatorname{Hom}_{\lie{gl}_d}
\left(
\lie N_\pi^{n-q-p}\otimes \lie q^{\otimes n-p}\otimes \lie{gl}_r^{\otimes p},
\,
\frac{
\operatorname{Hom}\big(\cJ_d^{\otimes_\pi n},\C\big)}
{
\operatorname{Hom}\big((\cJ_d^{\otimes_\pi n})^{\widehat\Delta_{\pi,k-1}},\C\big)}
\right)
\right).
\]
By repeating the previous argument, we obtain a complex that is quasi-isomorphic to
\[
\operatorname{Alt}_k
\left(
\frac{
\operatorname{Hom}
\left(
\cJ_d^{\otimes k}\otimes (\Lambda^\bullet\lie q^\vee)^{\otimes k},
\C
\right)}
{
\operatorname{Hom}
\left(
(\cJ_d^{\otimes k})^{\widehat\Delta_{k-1}}\otimes
(\Lambda^\bullet\lie q^\vee)^{\otimes k},
\C
\right)}
\otimes
\overline C^\bullet_{\mathrm{Lie}}(\lie w_d;\operatorname{Hom}(\Lambda^p\lie{gl}_r(\C\llbracket z\rrbracket),\C))^{\otimes k}
\right)^{\lie{gl}_d}.
\]
Observe that one factor in this tensor product can be identified with the $E_0$-page of the Serre–Hochschild spectral sequence associated to the inclusion $\lie{witt}_d\subset\lie{witt}_d \ltimes\lie{gl}_r(\cJ_d)$. This leads us to the following theorem.

\begin{theorem}\label{thm:localglr}
The total degree two part of the diagonal spectral sequence is contributed only by the local term $k=1$. Consequently,
\[
  \HH^2_{Lie}(\lie{witt}_d \ltimes\lie{gl}_r(\cJ_d))
  \cong
  H^2(C^\bullet_{(0)})
  \cong
  H^{2d+1}_{Lie}(\lie w_d \ltimes\lie{gl}_r(\C\llbracket z\rrbracket)).
\]
\end{theorem}

\paragraph
In \cite{Khoroshkin} it is shown that there is a graded ring isomorphism
\begin{equation}
  H^\bu_{Lie}(\lie{w}_d \ltimes \lie{gl}_r(\C\llbracket z\rrbracket) \; | \; \lie{gl}_d \oplus \lie{gl}_r) \simeq H^{\bu \leq
  2d} \left(B GL_d \times BGL_r\right) ,
\end{equation}
where the right hand side is the cohomology of the $2d$-skeleton of the product of classifying spaces.
In this theorem appears the relative Lie algebra cohomology with respect to the subalgebra
\begin{equation}\label{eqn:subalg}
  \lie{gl}_d \oplus \lie{gl}_r \subset \lie{w}_d \ltimes \lie{gl}_r(\C\llbracket z \rrbracket) 
\end{equation}
consisting of linear vector fields and constant $\lie{gl}_r$-valued power series.

Let $Y_{d,r}$ be the restriction of the universal bundle over the $2d$-skeleton of $B GL_d \times BGL_r$.
In terms of classifying spaces, the isomorphism of Khoroshkin takes the form
\begin{equation}
  H^\bu(\lie{w}_d \ltimes \lie{gl}_r(\llbracket z \rrbracket) \cong H^\bu (Y_{d,r})
\end{equation}
See \cite{BWac} for a sketch how to deduce this isomorphism from \cite{Khoroshkin}.
By the degeneration of the Serre spectral sequence, it follows that there is an isomorphism
\begin{equation}
  H_{Lie}^{2d+1}(\lie{w}_d \ltimes \lie{gl}_r(\llbracket z \rrbracket) \cong H^{2d+2} (BGL_d \times BGL_r) .
\end{equation}

Again, since dimensions match we can combine theorem \ref{thm:localglr} with theorem \ref{thm:cwdr} to see that 
\begin{equation}\label{}
  \mathsf{cw}_{d,r} \colon \C[y_1,\ldots,y_d,\lie{t}_1,\ldots,\lie{t}_r]_{d+1} \to \HH^2_{Lie}(\lie{witt}_d \ltimes
  \lie{gl}_r (\cJ_d))  
\end{equation}
is an isomorphism. 

\subsection{Relation to Lie algebra homology of differential operators}\label{ss:cohDiff}

In this section we characterize the image of the homomorphism
\[
  \HH^{Lie}_2(\lie{witt}_d)\to \HH^{Lie}_2(\cD_d)
\]
induced by the natural map of dg Lie algebras $\lie{witt}_d\to \cD_d$.

\begin{theorem}\label{thm:1dImage}
  The image of $\HH^{Lie}_2(\lie{witt}_d)$ in $\HH^{Lie}_2(\cD_d)$ is one-dimensional.
\end{theorem}

For $d=1$, it is a classical result that $\HH^{Lie}_2(\cD_1)$ is one-dimensional \cite{li1998central}. For $d=2$, we
proved in \cite{GWvir2d} that $\HH^{Lie}_2(\cD_2)$ is one-dimensional. 
We do not prove here that the full space
$\HH^{Lie}_2(\cD_d)$ is one-dimensional for arbitrary $d$. Instead, we prove the weaker statement above, which is sufficient for the local Grothendieck--Riemann--Roch theorem.

The proof proceeds in three steps. First, we compute a principal-symbol version of the relevant Lie algebra homology.
This is a Poisson approximation to the Lie algebra homology of differential operators with their commutator bracket.
Second, we reduce the principal-symbol computation to the homology of a Koszul complex. Finally, we use the order
filtration on differential operators to compare the principal-symbol calculation with the image of $\HH^{Lie}
_2(\lie{witt}_d)$ in $\HH^{Lie}_2(\cD_d)$.

\paragraph
Let
\[
  \lie{o}\define \C[z^{\1},\ldots,z^{\dd},p_{\1},\ldots,p_{\dd}]
\]
be the Poisson algebra of polynomial functions on $\T^*\AA^d$, equipped with the standard Poisson bracket
\[
  \{f,g\}
  =
  \sum_{\s=1}^d
  \left(
    \frac{\partial f}{\partial p_\s}\frac{\partial g}{\partial z^\s}
    -
    \frac{\partial f}{\partial z^\s}\frac{\partial g}{\partial p_\s}
  \right).
\]
Recall that the $(d-1)$st cohomology of the Jouanolou model $\cJ_d$ is
\[
  \cP_d\define H_{\dbar}^{d-1}(\cJ_d)
  \cong
  \op{S}(\lie q)\otimes \det(\lie q).
\]
Set
\[
  \lie n\define \C[p_{\1},\ldots,p_{\dd}]\otimes \cP_d.
\]
We regard $\lie n$ as the $(d-1)$st cohomology of the structure sheaf on $\T^*(\AA^d-\{0\})$. The Poisson algebra $\lie o$ acts on $\lie n$ by Hamiltonian vector fields:
\[
  f\cdot G
  =
  \sum_{\s=1}^d
  \left(
    \frac{\partial f}{\partial p_\s}\frac{\partial G}{\partial z^\s}
    -
    \frac{\partial f}{\partial z^\s}\frac{\partial G}{\partial p_\s}
  \right).
\]

\begin{lemma}\label{lem:principal-symbol-homology}
  The Lie algebra homology $\HH^{Lie}_d(\lie o;\lie n)$ is at most one-dimensional.
\end{lemma}

\begin{proof}
Let
\[
  \lie q^*=\op{span}\{z^1,\ldots,z^d\}\subset \lie o
\]
be the abelian Lie subalgebra spanned by the coordinate functions, and let
\[
  \lie u\define \lie o/\lie q^*.
\]
As a graded vector space, the Chevalley--Eilenberg complex decomposes as
\[
  C^{Lie}_\bullet(\lie o;\lie n)
  =
  \bigoplus_{a,b\geq 0}
  \lie n\otimes \wedge^a \lie u\otimes \wedge^b\lie q^* .
\]
We filter $C^{Lie}_\bullet(\lie o;\lie n)$ by the number of factors lying in $\lie u$. The associated graded has
\[
  E^0_{a,b}
  =
  \lie n\otimes \wedge^a\lie u\otimes \wedge^b\lie q^* .
\]
The differential on $E^0$ is the Chevalley--Eilenberg differential for the abelian Lie algebra $\lie q^*$ with
coefficients in $\lie n\otimes \wedge^a\lie u$. The action of $z^\s\in \lie q^*$ on $\lie n$ is
\[
  z^\s\cdot G = \frac{\partial G}{\partial p_\s},
\]
and the action on $\wedge^a\lie u$ is induced by the projected Poisson bracket.

For $a=0$, the $E^0$-complex is
\[
  \lie n\otimes \wedge^\bullet \lie q^*
  =
  \cP_d\otimes \C[p_\1,\ldots,p_\dd]\otimes \wedge^\bullet \lie q^* .
\]
This is the Koszul homology complex for the commuting operators
\[
  \frac{\partial}{\partial p_\1},\ldots,\frac{\partial}{\partial p_\dd}
\]
acting on $\C[p_\1,\ldots,p_\dd]$. Equivalently, after identifying $\C[p_\1,\ldots,p_\dd]$ with $\op{S}(\lie q)$, it is
the standard Koszul complex resolving the quotient by the images of these derivations. Hence its homology is
concentrated in cohomological degree $-d$, where it is
\[
  \cP_d\otimes \wedge^d\lie q^* .
\]

A similar conclusion holds for every $a\geq 0$. Indeed, write the action of $z^\s$ on $\lie n\otimes \wedge^a\lie u$
as
\[
  t^\s
  =
  -\frac{\partial}{\partial p_\s}\otimes 1
  +
  1\otimes \theta^\s,
\]
where $\theta^\s$ is the action of $z^\s$ on $\wedge^a\lie u$. The operators $\theta^\s$ commute with one another and are locally nilpotent. 
Therefore the exponential operator $\exp\left(-\sum_{\s=1}^d p_\s\theta^\s\right)$
conjugates the operators $t^\s$ to
\[
  \exp\left(-\sum_{\s=1}^d p_\s\theta^\s\right) t^{\s} 
  \exp\left(\sum_{\s=1}^d p_\s\theta^\s\right) = 
  -\frac{\partial}{\partial p_\s}\otimes 1 .
\]
Thus, for every $a\geq 0$ and $b < d$ we have $E^1_{a,b}=0$ and
\[
  E^1_{a,d}
  \cong
  \cP_d\otimes \wedge^a\lie u\otimes \wedge^d\lie q^* .
\]

We now examine the $E^1$-differential $\d^1 \colon E^1_{1,d}\to E^1_{0,d}$.
We claim that
\begin{equation}\label{eq:image}
  \op{im} \left(\d^1|_{E_{1,d}^1}\right) \supset \left(\sum_{\s=1}^d \frac{\partial \cP_d}{\partial z^\s}\right)
  \otimes
  z^\1\wedge\cdots\wedge z^\dd .
\end{equation}
Fix $G\in \cP_d$ and an index $\s$. 
The class of $G\otimes p_\s\in \cP_d\otimes \lie u$ has a representative on the $E^0$-page given by
\[
  \eta^\s
  \define
  \exp\left(-\sum_{\ii=1}^d p_\ii \theta^\ii\right)
  (G\otimes p_\s)
  \otimes
  z^\1\wedge\cdots\wedge z^\dd .
\]
In fact, this exponential truncates after the first correction term and so equivalently
\[
  \eta^\s
  =
  \left(
    G\otimes p_\s
    -
    G p_\s\otimes 1
  \right)
  \otimes
  z^\1\wedge\cdots\wedge z^\dd .
\]
By construction, $\eta^\s$ is a cycle for the $E^0$ differential.
The Hamiltonian action of the linear function $p_{\s}$ on $\lie{n}$ is
\[
  p_\s\cdot G
  =
  \frac{\partial G}{\partial z^\s}.
\]
Therefore, the $E^1$-differential applied to $\eta^\s$ is
\[
  \d^1 (\eta^\s)
  =
  \frac{\partial G}{\partial z^\s}
  \otimes
z^\1\wedge\cdots\wedge z^\dd .
\]
This proves the claim \eqref{eq:image}.

It follows that $E^2_{0,d}$ is a quotient of
\[
  \frac{
    \cP_d\otimes z^\1\wedge\cdots\wedge z^\dd
  }{
    \left(\sum_{\s=1}^d \frac{\partial \cP_d}{\partial z^\s}\right)
    \otimes z^\1\wedge\cdots\wedge z^\dd
  }.
\]
Since
\[
  \frac{\cP_d}{\sum_{\s=1}^d \partial_{z^\s}\cP_d}
  \cong
  \C\cdot P(d),
\]
we conclude that $E^2_{0,d}$ has dimension at most one. Because the spectral sequence is concentrated in the region $b\leq d$ and all terms with $b<d$ vanish on the $E^1$-page, every class in total degree $d$ is detected on the line $b=d$. Hence
\[
  \dim \HH^{Lie}_d(\lie o;\lie n)\leq 1.
\]
\end{proof}

\paragraph
Consider the algebra of polynomial differential operators $\lie{D}_d$ on $\AA^d$.
There is an increasing filtration on $\lie{D}_d$ given by the order of differential operators and its associated graded is the Poisson
algebra $\lie{o}$.
With this identification, the principal symbol of $\del_{\s}$ is $p_\s$ and the commutator induces the standard Poisson bracket.

The same filtration applies to the Jouanolou model of differential operators. At the level of $\dbar$-cohomology, one obtains a filtered Lie algebra $H^0_{\dbar}(\cD_d)$ and a filtered module $H^{d-1}_{\dbar}(\cD_d)$, with
\[
  \op{gr} H^\bu_{\dbar}(\cD_d) \cong \op{gr}\left(
    H^0_{\dbar}(\cD_d)\ltimes H^{d-1}_{\dbar}(\cD_d)[-d+1]
  \right)
  \cong
  \lie o\ltimes \lie n[-d+1].
\]

\begin{lemma}\label{lem:D-homology}
  One has
  \[
    \dim
    \HH^{Lie}_d
    \left(
      H^0_{\dbar}(\cD_d);
      H^{d-1}_{\dbar}(\cD_d)
    \right)
    \leq 1.
  \]
\end{lemma}

\begin{proof}
Filter the Chevalley--Eilenberg complex
\[
  C^{Lie}_\bullet
  \left(
    H^0_{\dbar}(\cD_d);
    H^{d-1}_{\dbar}(\cD_d)
  \right)
\]
by total order of differential operators. Explicitly, a chain
\[
  M\otimes D_1\wedge\cdots\wedge D_k
\]
has order at most $r$ if the sum of the order of $M$ and the orders of the $D_i$ is at most $r$. The Chevalley--Eilenberg differential preserves this filtration, because the commutator lowers order by one while the module action has the corresponding principal-symbol action on the associated graded.

The associated graded complex is therefore
\[
  \op{gr} C^{Lie}_\bullet
  \left(
    H^0_{\dbar}(\cD_d);
    H^{d-1}_{\dbar}(\cD_d)
  \right)
  \cong
  C^{Lie}_\bullet(\lie o;\lie n).
\]
Thus the order filtration gives a spectral sequence whose $E^1$-page is
\[
  E^1
  \cong
  \HH^{Lie}_\bullet(\lie o;\lie n).
\]

By lemma \ref{lem:principal-symbol-homology}, the group $\HH^{Lie}_d(\lie o;\lie n)$ is at most one-dimensional. More
precisely, the proof of this lemma shows that any class in degree $d$ is represented by the order-zero class
\[
  P(d)\otimes z^\1\wedge\cdots\wedge z^\dd .
\]
Indeed, on the $E^1$-page of the auxiliary spectral sequence used there, the only possible degree-$d$ contribution lies in bidegree $(0,d)$, and after quotienting by the image of $d^1$ it is a quotient of
\[
  \frac{\cP_d}{\sum_{\s=1}^d \partial_{z^\s}\cP_d}
  \otimes \wedge^d\lie q^*
  \cong
  \C\cdot P(d)\otimes \wedge^d\lie q^*.
\]
This representative has no $p$-dependence, hence has differential-operator order zero.

Now let $\alpha$ be a closed element of
\[
  C^{Lie}_d
  \left(
    H^0_{\dbar}(\cD_d);
    H^{d-1}_{\dbar}(\cD_d)
  \right).
\]
and that $\alpha$ has positive total order $r>0$.
Let $\sigma_r(\alpha)$ denote its leading symbol. 
Since $\alpha$ is closed, $\sigma_r(\alpha)$ is a cycle in $C^{Lie}_d(\lie o;\lie n)$ of order $r$. 
But the degree-$d$ principal-symbol homology has no nonzero classes of positive order. 
Therefore $\sigma_r(\alpha)$ is a boundary:
\[
  \sigma_r(\alpha)=\d \overline{\beta}
\]
for some $\overline{\beta}\in C^{Lie}_{d+1}(\lie o;\lie n)$ of order $r$. Choose a lift $\beta$ of $\overline{\beta}$ to the filtered Chevalley--Eilenberg complex. 
Then
\[
  \alpha-\d \beta
\]
is homologous to $\alpha$ and has strictly smaller total order.

Repeating this procedure, we may replace $\alpha$ by a homologous cycle of order zero. Therefore every homology class in degree $d$ is represented by a chain of the form
\[
  G\otimes f_1\wedge\cdots\wedge f_d,
  \qquad
  G\in \cP_d,\quad f_i\in \C[z] .
\]
Passing again to the principal-symbol calculation, the class of such a cycle is a scalar multiple of
\[
  P(d)\otimes z^\1\wedge\cdots\wedge z^\dd .
\]
Consequently the filtered homology group
\[
  \HH^{Lie}_d
  \left(
    H^0_{\dbar}(\cD_d);
    H^{d-1}_{\dbar}(\cD_d)
  \right)
\]
is at most one-dimensional.
\end{proof}

\paragraph{Proof of theorem \ref{thm:1dImage}}
The map $\lie{witt}_d\to \cD_d$
sends a polynomial vector field
to the corresponding first-order differential operator.
After passing to $\dbar$-cohomology, the image of the relevant degree-two Lie homology class is represented, in the associated graded, by the principal-symbol class
\[
  P(d)\otimes z^\1\wedge\cdots\wedge z^\dd .
\]
By Lemma \ref{lem:D-homology}, the target space containing this image is at most one-dimensional in the relevant summand. Hence the image of
\[
  \HH^{Lie}_2(\lie{witt}_d)\to \HH^{Lie}_2(\cD_d)
\]
has dimension at most one.

It remains to show that the image is nonzero. 
This follows from the nontriviality of the Fuks--Feigin cocycle: under the symbol map, the class constructed from $\lie{witt}_d$ pairs nontrivially with the principal-symbol class
\[
  P(d)\otimes z^\1\wedge\cdots\wedge z^\dd .
\]
By \cite{FHK}, this class is nontrivial. Therefore the image is nonzero.

Combining the upper bound with nontriviality, the image is exactly one-dimensional.

\paragraph
Using the same strategy, we can prove the following theorem.
  \begin{theorem}\label{1dImageMatrix}
     The image of $\HH_2^{Lie}(\lie{witt}_d\ltimes \lie{gl}_r(\cJ_d))$ inside of $\HH^{Lie}_2(\lie{gl}_r(\cD_d))$ is one-dimensional.
     \end{theorem}
\begin{proof}
Recall that in definition \ref{dfn:partitionvv}, we formed a collection of cycles on the first page of the $\dbar,\mathbf{d}^{(2)}$ spectral sequence, which can be extended to full cycles that span the entire homology. Thus we only need to show that the image of 
     \[
\mathbb{v}_{\lambda,\theta_{(k_1,\ldots,k_r)}}(d)
:= \sum_{\tau \in S_k} \sum_{\ell=1}^{n_{(k_1,\ldots,k_r)}} 
P(d)\,
A^{\theta_{(k_1,\ldots,k_r)},\ell}_{\tau(0)} \wedge z^{\1} A^{\theta_{(k_1,\ldots,k_r)},\ell}_{\tau(1)} \wedge \cdots \wedge z^{\kkk} A^{\theta_{(k_1,\ldots,k_r)},\ell}_{\tau(k)}
\]
\[
\wedge \,\det(z^{\Lambda_1},\mathsf{Eu}_{\Lambda_1}) \wedge \cdots \wedge \det(z^{\Lambda_{d-k}},\mathsf{Eu}_{\Lambda_{d-k}}),
\]
    is a scalar multiple of 
    \[
    P(d)\mathbf{I}_r\otimes z^{\1}\mathbf{I}_r\wedge\cdots \wedge z^{\dd}\mathbf{I}_r.
    \]
    Consider the following elements
     \[
\sum_{\tau \in S_k} \sum_{\ell=1}^{n_{(k_1,\ldots,k_r)}} 
P(d)\,
A^{\theta_{(k_1,\ldots,k_r)},\ell}_{\tau(0)} \wedge z^{\1}\partial_{z^{\1}}\cdot A^{\theta_{(k_1,\ldots,k_r)},\ell}_{\tau(1)} \wedge  z^{\1}\mathbf{I}_r\wedge\cdots \wedge z^{\kkk} A^{\theta_{(k_1,\ldots,k_r)},\ell}_{\tau(k)}
\]
\[
\wedge \,\det(z^{\Lambda_1},\mathsf{Eu}_{\Lambda_1}) \wedge \cdots \wedge \det(z^{\Lambda_{d-k}},\mathsf{Eu}_{\Lambda_{d-k}}).
\]
 By a direct computation, it is a Chevalley-Eilenberg homotopy between $\mathbb{v}_{\lambda,\theta_{(k_1,\ldots,k_r)}}(d)$ and 
     \[
\sum_{\tau \in S_k} \sum_{\ell=1}^{n_{(k_1,\ldots,k_r)}} 
P(d)\,
A^{\theta_{(k_1,\ldots,k_r)},\ell}_{\tau(0)}A^{\theta_{(k_1,\ldots,k_r)},\ell}_{\tau(1)}  \wedge z^{\1}\mathbf{I}_r \wedge \cdots \wedge z^{\kkk} A^{\theta_{(k_1,\ldots,k_r)},\ell}_{\tau(k)}\wedge \,\det(z^{\Lambda_1},\mathsf{Eu}_{\Lambda_1}) \wedge \cdots \wedge \det(z^{\Lambda_{d-k}},\mathsf{Eu}_{\Lambda_{d-k}}),
\]
 Iterating this argument, we obtain that $\mathbb{v}_{\lambda,\theta_{(k_1,\ldots,k_r)}}(d)$ is homologous to 
         \[
\sum_{\tau \in S_k} \sum_{\ell=1}^{n_{(k_1,\ldots,k_r)}} 
P(d)\,
A^{\theta_{(k_1,\ldots,k_r)},\ell}_{\tau(0)}A^{\theta_{(k_1,\ldots,k_r)},\ell}_{\tau(1)}\ldots A^{\theta_{(k_1,\ldots,k_r)},\ell}_{\tau(k)}  \wedge z^{\1}\mathbf{I}_r \wedge \cdots \wedge z^{\kkk}\mathbf{I}_r \wedge \,\det(z^{\Lambda_1},\mathsf{Eu}_{\Lambda_1}) \wedge \cdots \wedge \det(z^{\Lambda_{d-k}},\mathsf{Eu}_{\Lambda_{d-k}}),
\]
    and, using the fact that $\lie{sl}_r=[\lie{sl}_r,\lie{sl}_r]$, this class is in turn homologous to
     \[
\sum_{\tau \in S_k} \sum_{\ell=1}^{n_{(k_1,\ldots,k_r)}} 
P(d)\,
\mathrm{tr}_{\lie{gl}_r}\big(A^{\theta_{(k_1,\ldots,k_r)},\ell}_{\tau(0)}A^{\theta_{(k_1,\ldots,k_r)},\ell}_{\tau(1)}\ldots A^{\theta_{(k_1,\ldots,k_r)},\ell}_{\tau(k)} \big) \wedge z^{\1}\mathbf{I}_r \wedge \cdots \wedge z^{\kkk}\mathbf{I}_r \wedge \,\det(z^{\Lambda_1},\mathsf{Eu}_{\Lambda_1}) \wedge \cdots \wedge \det(z^{\Lambda_{d-k}},\mathsf{Eu}_{\Lambda_{d-k}}).
\]
Finally we apply theorem \ref{thm:1dImage} to complete the proof.
  \end{proof}

\section{Local Grothendieck-Riemann-Roch theorem}\label{s:GRR}

In this final section we formulate and prove a local, universal form of the Grothendieck-Riemann-Roch theorem for
formal families of complex varieties.
It is a natural generalization of the relationship of the Mumford determinant line bundle defined on the
moduli space of Riemann surfaces and the central charge. 

The set up is as follows.
First, note that $\cJ_d$ is a representation for the dg lie algebra $\lie{witt}_d$
\[
 \sfL \colon \lie{witt}_d \xto{\simeq} \op{Der}( \cJ_d) \to \op{End} \cJ_d .
\]
Pulling back along $\sfL$ is a map in Lie algebra hyper Lie algebra cohomology 
\begin{equation}\label{}
  \sfL^* \colon \HH^2_{Lie}(\op{End} \cJ_d) \to \HH^2_{Lie} (\lie{witt}_d) . 
\end{equation}

More generally, let $\cV$ be any tensor $\lie{witt}_d$-module (recall \ref{TensorMod}). 
Then, we obtain a homomorphism (essentially Lie derivative): 
\begin{equation}\label{}
  \sfL_{V} \colon \lie{witt}_d \to \op{End}(\cV)  
\end{equation}
and hence a map in cohomology 
\begin{equation}\label{}
  \sfL_V^* \colon \HH^2_{Lie}(\op{End} \cV) \to \HH^2_{Lie}(\lie{witt}_d) .
\end{equation}

Assume that $V$ is an $r$-dimensional tensor $\lie{gl}_d$-representation. 
Then, $\rho_V \colon \lie{gl}_d \to \op{End}(V)\simeq \lie{gl}_r$ induces a morphism
\[
\rho_V^*: \op{S}^\bu(\lie{gl}_r^*)^{\lie{gl}_r}\simeq \C[\lie{t}_1,\ldots,\lie{t}_r]\longrightarrow \op{S}^\bu(\lie{gl}_d^*)^{\lie{gl}_d}\simeq \C[y_1,\ldots,y_d].
\]
We define $\op{Ch}_i(\cV) \define \rho_V^*(\lie{t}_i)\in \C[y_1,\ldots,y_d]$. With this in place, the polynomial
\[
\op{Td}\,\op{Ch}(\cV)\in \C[y_1,\ldots,y_d]
\]
is now well-defined.

By \cite[theorem 5.4.16]{FHK} the dg vector space $\op{End} \cV$ is a dg Tate vector space, and therefore its second Lie algebra
cohomology is one-dimensional, with generator we denote 
\begin{equation}\label{}
  \mathfrak{t} \in \HH^2_{Lie}(\op{End}\cV) \simeq \C.
\end{equation}
It is dual to the homology class
\begin{align}\label{eq:universalV}
  \left[z^\1 \wedge \cdots \wedge z^{\ddd} \wedge P \right]  &  \in 
  \HH_2^{Lie}( \op{End} \cV) \\ 
  \<\mathfrak{t}, [z^\1 \wedge \cdots \wedge z^{\ddd} \wedge P] \> & = 1 .
\end{align}
Here, we use that $z^\1, \ldots, z^{\ddd}, P \in \cJ_d$ acts purely through the $\cJ_d$-module structure.
With this setup, we will prove the following local, universal version of the Grothendieck-Riemann-Roch theorem which
was stated as a
conjecture in \cite{KapranovNotes}.

\begin{theorem}\label{thm:grr} 
  Let $V$ be a finite-dimensional $\lie{gl}_d$-representation.
  The composition
\begin{equation}\label{}
  \HH_2^{Lie}(\lie{witt}_d) \xto{(\sfL_V)_*} \HH_2^{Lie}( \op{End} \cV) \xto{\tau} \C ,
\end{equation}
is the following universal class:
  \begin{equation}\label{}
    \op{Td} \op{Ch}(\cV)|_{2d+2} \in \C[\op{ch}_1,\ldots,\op{ch}_d]_{2d+2}
  \end{equation}
  In other words, at the level of cohomology, one has 
  \begin{equation}\label{}
    \sfL^*_V [z^\1 \wedge \cdots \wedge z^{\ddd} \wedge P] = \left[\op{Td} \op{Ch}(\cV)|_{2d+2}\right] \in \HH^2_{Lie}(\lie{witt}_d) .
  \end{equation}
\end{theorem}
%\gui{I think we need to clarify the meaning of Ch for tensor module}

\subsection{Differential operator formulation} 

We sketch our method of proof of theorem \ref{thm:grr} before getting into details.
The main idea is to involve the Jouanolou model of the (dg) algebra of differential operators $\cD_d$ on punctured affine space $\AA^d - \{0\}
$, see \cite[section 1.5]{GWvir2d}.
By construction, $\sfL \colon \lie{witt}_d \ltimes \lie{gl}_r(\cJ_d) \to \op{End} \cJ_d^{\oplus r}$ factors through differential operators
\begin{equation}\label{}
  \lie{witt}_d \ltimes \lie{gl}_r(\cJ_d)\xto{i} \lie{gl}_r(\cD_d) \xto{\til \sfL} \op{End} \cJ_d ^{\oplus r}
\end{equation}
where the first map simply views vector fields acting by first-order differential operators and the second map is the natural inclusion.
From here, we will use theorem \ref{1dImageMatrix} together with the existence of an algebraic trace in the spirit of
\cite{FFS}.
Indeed, in the next section we will define the ``residue trace" on our Jouanolou model for differential operators as a degree two class 
\begin{equation}\label{eq:trD}
  \op{ResTr}_{\cD}^{S^1} \in \HH^2_{Lie}(\lie{gl}_r(\cD_{d})) . 
\end{equation}
Consider the commutative diagram 
\begin{equation}\label{eq:diagram}
  \begin{tikzcd}
    \HH_2^{Lie}(\lie{witt}_d\ltimes \lie{gl}_r(\cJ_d)) \ar[rr,"\sfL_*"] \ar[dr,"i_*"] & & \HH_2^{Lie}(\op{End} \cJ_d^{\oplus r}) \ar[r,"\tau","\cong"'] & \C \\
                                                                     & \HH_2^{Lie}(\lie{gl}_r(\cD_d)) \ar[rr,"\op{ResTr}^{S^1}_{\cD}",
    ""'] \ar[ur,"\til \sfL_*"] & &
                                                                 \C 
                                                                 \end{tikzcd}
                                                               \end{equation}
By definition $\tau([z^1 \wedge z^2 \wedge \cdots \wedge z^d \wedge P]) = 1$
where $\tau$ is the universal Tate class of \cite{FHK}.
We will also show that $\op{ResTr}_{\cD}^{S^1}([z^1 \wedge \cdots \wedge z^d \wedge P]) = 1$.

\begin{theorem}\label{thm:grrD}
The composition 
\begin{equation}\label{}
  \HH_2^{Lie}(\lie{witt}_d\ltimes \lie{gl}_r(\cJ_d)) \xto{i_*} \HH_2^{Lie}(\lie{gl}_r(\cD_d)) \xto{\op{ResTr}_{\cD}^{S^1}} \C .
\end{equation}
is $\op{Td}\cdot \op{Ch}|_{2d+2}$.
\end{theorem}

Theorem \ref{thm:grr} can be reduced to the theorem above. 
Indeed, let $\cV= \cJ_d \otimes V$ be a tensor module. 
The $\lie{witt}_d$-action on $\cV$ factors as a composition
\[
\lie{witt}_d \to \lie{witt}_d \ltimes \lie{gl}_r(\cJ_d)\simeq \lie{witt}_d \ltimes (\operatorname{End}(V)\otimes \cJ_d)\to \op{End}(\cV).
\]
Thus, Theorem \ref{thm:grr} can be deduced from the above theorem.

\subsection{Algebraic residue trace}
As we have already mentioned, this class \eqref{eq:trD} is inspired by the formal algebraic trace map of \cite{FFS}.
We introduce the dg Weyl algebra $\cW_d$ associated to the punctured formal disk, see section \cite[section 1.5]{GWvir2d} for definitions and
notations.
There is a dg algebra isomorphism from $\mathcal{D}_{d}$ to $\mathcal{W}_{d}$ given by the symbol map
\[
\sigma\left(f(z^{\1},...,z^{\dd})F(\partial_{z^{\1}},...,\partial_{z^{\dd}})\right) \define e^{-\frac{1}{2}\sum^d\limits_{\s=1}\partial_{p^{\s}}\partial_{q^{\s}}}(f(q^{\1},...,q^{\dd})F(p^{\1},...,p^{\dd})).
\]
In particular
\[
\sigma\left(T\right)=T-\frac{1}{2}\mathrm{div}(T)
\]
where $T\in \lie{witt}_d \subset \cD_d$ and $\op{div}(\sum_i f^i \del_i) = \sum_i \del_i f^i \in \cJ_d$.

The residue trace is most fundamentally defined at the level of cyclic cohomology.
It is a \textit{non-commutative} version of the cyclic class $\rho \in \HH_\lambda^1(\cJ_d)$ of \cite{FHK}. We follow the presentation in \cite[3]{GLX}, which gives a convenient setting to generalize \cite{FFS} to our case.

\begin{definition}
    Define the residue trace cochain
    \[
      \op{ResTr}^{S^1}_\lambda \in  \overline{C}_{\lambda}^{\bu}(\lie{gl}_r(\mathcal{W}_d))^1
  \]
   of total degree $+1$ by the formula
\begin{multline}
  \mathrm{ResTr}_\lambda^{S^1}\left(O_0\otimes \cdots\otimes O_k\right) =
  \mathrm{Res}_{q} \left({\int_{S^1_{\mathrm{cyc}}[k+1]}\operatorname{tr}_{\lie{gl}_r}\circ\mathbf{Mult}\left(e^{\Pi+D}(O_0\cdot d\theta_0\otimes O_1\otimes\cdots\otimes O_k)\right)}|_{p=0}\right),
\end{multline}
%\begin{multline}
     %mathrm{Tr}^{S^1}\left(O_0\otimes \cdots\otimes O_k\right) \\
%=\mathrm{Res}\left({\int_{\triangle_k}e^{\sum^d\limits_{\s=1}\sum\limits_{i<j}\psi(\theta_i-\theta_j)((\partial_{q^{\s}})_i\otimes(\partial_{p^{\s}})_j-(\partial_{p^{\s}})_i\otimes(\partial_{q^{\s}})_j)}(O_0\otimes D(O_1)\otimes\cdots\otimes D(O_k))}|_{p=0}\right),
%\end{multline}
    where:
    \begin{itemize}
    \item[(i)] the operators $\Pi$ and $D$ are as follows:
      \[
\Pi\define \frac{1}{2}\sum_{i<j}\pi^*_{ij}\SP\cdot \Pi_{ij},\quad \Pi_{ij}:=\sum^d_{\s=1}(\partial_{q^{\s}})_i\otimes(\partial_{p^{\s}})_j-(\partial_{p^{\s}})_i\otimes(\partial_{q^{\s}})_j,
\]
\[
D(O_0\cdot d\theta_0\otimes O_1\otimes\cdots\otimes O_k)=\sum^{k}_{i=1}O_0\cdot d\theta_0\otimes \cdots\otimes D_i(O_i)\otimes\cdots\otimes O_k,
\]
    \[
    D_i(O_i)\define \sum^d_{\s=1}{\partial_{q^{\s}}}(O_i)dq^{\s}\cdot d\theta_i+\sum^d_{\s=1}{\partial_{p^{\s}}}(O_i)dp^{\s}\cdot d\theta_i.
  \]
      \item[(ii)] Let $\mathrm{Cyc}_{S^1}[k+1]$ be the configuration space of $k+1$ anti-clockwise cyclic ordered points on the circle $S^1$. We have a natural identification
      \[
      \mathrm{Cyc}_{S^1}[k+1]\simeq S^1\times \mathring{\Delta}_{k}
      \]
      \[
      (\theta_0,\dots,\theta_k)\rightarrow \{p_0\}\times (\theta_{0,1},\theta_{1,2},\dots,\theta_{k,0}).
      \]
      Here we denote $\Delta_k=\{(\theta_{0,1},\theta_{1,2},\dots,\theta_{k-1,k},\theta_{k,0})\in [0,1]^{k+1}| \theta_{0,1}+\theta_{1,2}+\dots+\theta_{k,0}=1\}$ viewed as a manifold with corners and $\mathring{\Delta}_{k}$ its interior. This allows us to compactify $\mathrm{Cyc}_{S^1}[k+1]$ by $S^1\times\Delta_{k}$, which will be denoted by $S^1_{\mathrm{cyc}}[k+1]$.  
      %$\psi:\R\rightarrow[0,1]$ is the function defined by $\psi(u)=2u+1$ if $-1\leq u <0$ and $\psi(u+1)=\psi(u)$. 
      \item[(iii)] For $0\leq i<j\leq k$, we have the fogetful map
      \[
\pi_{ij}:      S^1_{\mathrm{cyc}}[k+1]\rightarrow       S^1_{\mathrm{cyc}}[2]
      \]
      \[
      (\theta_0,\dots,\theta_k)\rightarrow (\theta_i,\theta_j).
      \]
    Recall  that $S^1_{\mathrm{cyc}}[2]\simeq S^1\times \Delta_1=S^1\times [0,1]$. The function $\SP$ on $S^1_{\mathrm{cyc}}[2]$ is define to be $u-\frac{1}{2}$ where $u\in \Delta_1=[0,1]$.
    
      \item[(iv)] The notation $\mathbf{Mult}\left(-\right)|_{p=0}$ means that we first multiply together the tensor components and then set $p^{\s}=dp^{\s}=0$.
      \item[(v)] $\op{Res}_q$ is the same Jouanolou residue introduced in section \cite[section 2.1]{GWvir2d} except we use the
        variable $q$ instead of $z$.
      \end{itemize}
\end{definition}
From the definition, $\mathrm{ResTr}^{S^1}$ is zero except $k=d$.

\subsection{One-loop result}
Recall the logarithmic Todd polynomial is
\[
\log(\mathrm{Todd}(y))=\frac{y}{2}-\sum_{m\geq 1}\frac{B_{2m}}{2m}y^m
\]
where $B_{2m}$ is $2m$th Bernoulli number.
We proceed to the proof of theorem \ref{thm:grrD}.
\begin{theorem*}[\ref{thm:grrD}]
  One has 
  \begin{equation}\label{}
    \op{ResTr}_{\cD}^{S^1}|_{\lie{witt}_d} = \op{Td}\op{Ch}|_{2d+2}
  \end{equation}
\end{theorem*}

\begin{proof}
 The proof is a straightforward refinement of the argument in \cite{GWWchiral}, which employs a setup similar to that of \cite{FFS} and \cite{GLX}. The main distinction is that we introduce a twist by the symbol map (a secondary difference is that we work with the residue rather than formal geometry). The central idea is to rewrite the exponential in a Wick-type form, which in turn produces a Feynman diagram expansion for the residue trace. Furthermore, because we restrict attention to vector fields instead of general higher-order differential operators, the resulting expansion involves only one-loop graphs.

The term involving $\lie{gl}_r(\cJ_d)\subset\lie{witt}_d\ltimes \lie{gl}_r(\cJ_d)$ cannot be connected by any Wick contraction, since only one-loop diagrams contribute, and contracting the $\lie{gl}_r(\cJ_d)$ part would produce a tree component. Consequently, taking the trace $\operatorname{tr}_{\lie{gl}_r}$ yields the factor $\op{Ch}$. This completes the proof.

\end{proof}

\newpage 

\printbibliography
%\bibliographystyle{amsplain}
%\bibliography{RRChiral}
\end{document}